\documentclass{amsart}
\usepackage{amsthm, amsmath, amssymb}
\usepackage{hyperref}

\newcommand{\mc}{\mathcal}
\newcommand{\ph}{\varphi}
\newcommand{\qbinom}[2]{\genfrac{[}{]}{0pt}{}{#1}{#2}}

\newtheorem{thm}{Theorem}
\newtheorem{lem}{Lemma}
\newtheorem{prop}{Proposition}
\newtheorem{cor}{Corollary}
\newtheorem{conj}{Conjecture}

\theoremstyle{definition}
\newtheorem{defin}{Definition}

\newtheorem{example}{Example}
\newtheorem{prob}{Problem}

\theoremstyle{remark}
\newtheorem{remark}{Remark}

\DeclareMathOperator{\spn}{span}

\usepackage{todonotes}

\title{Separable elements and splittings of Weyl groups}

\author{Christian Gaetz}
\thanks{C.G. is partially supported by a National Science Foundation Graduate Research Fellowship under Grant No. 1122374.}
\address{Department of Mathematics, Massachusetts Institute of Technology, Cambridge, MA 02139}
\email{\href{mailto:gaetz@mit.edu}{{\tt gaetz@mit.edu}}}
\author{Yibo Gao}
\email{\href{mailto:gaoyibo@mit.edu}{{\tt gaoyibo@mit.edu}}}

\date{\today}

\begin{document}

\begin{abstract}
We continue the study of \emph{separable elements} in finite Weyl groups, introduced in \cite{First-separable-paper}.  These elements generalize the well-studied class of separable permutations.  We show that the multiplication map $W/U \times U \to W$ is a length-additive bijection or \emph{splitting} of the Weyl group $W$ when $U$ is an order ideal in right weak order generated by a separable element, where $W/U$ denotes the \emph{generalized quotient}.  This generalizes a result for the symmetric group, answering an open problem of Wei \cite{Wei2012Product}.

For a generalized quotient of the symmetric group, we show that this multiplication map is a bijection if and only if $U$ is an order ideal in right weak order generated by a separable element, thereby classifying those generalized quotients which induce splittings of the symmetric group, resolving a problem of Bj\"{o}rner and Wachs from 1988 \cite{Bjorner-Wachs}.
We also prove that this map is always surjective when $U$ is an order ideal in right weak order.  Interpreting these sets of permutations as linear extensions of 2-dimensional posets gives the first direct combinatorial proof of an inequality due originally to Sidorenko in 1991, answering an open problem Morales, Pak, and Panova \cite{PPM}.  We also prove a new $q$-analog of Sidorenko's formula.  All of these results are conjectured to extend to arbitrary finite Weyl groups.

Finally, we show that separable elements in $W$ are in bijection with the faces of all dimensions of several copies of the graph associahedron of the Dynkin diagram of $W$.  This correspondence associates to each separable element $w$ a certain \emph{nested set}; we give product formulas for the rank generating functions of the principal upper and lower order ideals generated by $w$ in terms of these nested sets, generalizing several known formulas.

\end{abstract}

\maketitle

\section{Introduction} \label{sec:intro}

A permutation $w=w_1\ldots w_n$ is \emph{separable} if it avoids the patterns 3142 and 2413, meaning that there are no indices $i_1<i_2<i_3<i_4$ such that the values $w_{i_1}w_{i_2}w_{i_3}w_{i_4}$ are in the same relative order as 3142 or 2413.  This well-studied class of permutations arose in the study of pop-stack sorting \cite{Avis} and has found applications in algorithmic pattern matching and bootstrap percolation \cite{Bose, Shapiro}.  These permutations have a remarkable recursive combinatorial structure and are enumerated by the Schr\"{o}der numbers \cite{West}. 

This paper\footnote{An extended abstract of this work has been submitted to the proceedings of FPSAC 2020.} is a sequel to \cite{First-separable-paper}.  Whereas that paper was concerned with defining separable elements in arbitrary finite Weyl groups and establishing some of their structural properties (such as their characterization by root system pattern avoidance) this paper is concerned with certain algebraic decompositions of the Weyl group induced by separable elements and with applying results about these decompositions to resolve several open problems.  In addition, we show that the combinatorics of separable elements is closely linked with the combinatorics of graph associahedra.  Most of our results are new even for the case of the symmetric group.

Bj\"{o}rner and Wachs \cite{Bjorner-Wachs} introduced the notion of a \emph{generalized quotient} $W/U$ in a Coxeter group $W$, where $U \subseteq W$ is an arbitrary subset:
\[
W/U:=\{w \in W \: | \: \ell(wu)=\ell(w)+\ell(u), \forall u \in U\}.
\]
They proved that $W/U$ is always an interval $[e,w_0u_0^{-1}]_L$ in the \emph{left weak order}, where $u_0$ is the least upper bound of $U$ in the \emph{right weak order}.  When $U=W_J$ is a parabolic subgroup, the generalized quotient $W/U$ is precisely the \emph{parabolic quotient} $W^J$.  It is well known that the multiplication map $W^J \times W_J \to W$ is a length-additive bijection.  Any such pair $(X,Y)$ of subsets of $W$ for which the multiplication map $X \times Y \to W$ is a length-additive bijection is called a \emph{splitting}.

In Section \ref{sec:background} we recall background on Weyl groups, root systems, and the weak order which is not specific to the study of separable elements.  

Section \ref{sec:separable} defines the notion of a separable element in a finite Weyl group $W$ and states some results from \cite{First-separable-paper} which will be needed later.

Section \ref{sec:gen-quotients} states our three main results about generalized quotients.  First, in Theorem \ref{thm:separable-splitting} we show that there is a splitting $W/[e,u]_R \times [e,u]_R \to W$ when $u$ is separable, answering an open problem of Wei \cite{Wei2012Product}.  Next, in the case $W=S_n$, we show in Theorem \ref{thm:only-if} that any splitting $X \times Y \to W$ is of this form; this solves a problem of Bj\"{o}rner and Wachs \cite{Bjorner-Wachs} from 1988.  Lastly, in Theorem \ref{thm:surjectivity} we show that the multiplication map $W/[e,u]_R \times [e,u]_R \to W$ is surjective for \emph{any} $u \in W=S_n$.  Together with the discussion in Section \ref{sec:linear-extensions}, this resolves an open problem of Morales, Pak, and Panova \cite{PPM}.  In Section \ref{sec:linear-extensions} we also give a new $q$-analog of an inequality for linear extensions of 2-dimensional posets due to Sidorenko \cite{Sidorenko}. In Section \ref{sec:conjectures} we conjecture that Theorems \ref{thm:only-if} and \ref{thm:surjectivity} extend to arbitrary finite Weyl groups.

In Section \ref{sec:associahedra} we give an elegant bijection between separable elements $u \in W$ and \emph{nested sets} $\mc{N}_u$ on $\Gamma$, the Dynkin diagram associated to $W$.  By a result of Postnikov \cite{beyond-paper}, these nested sets index the faces of the \emph{graph associahedron} of $\Gamma$.  We give a product formula for the rank generating functions of $[e,u]_L$ and $[e,u]_R$ in terms of the nested set $\mc{N}_u$; this formula generalizes several formulas in the literature.  


Finally, Sections \ref{sec:proof-of-only-if} and \ref{sec:proof-of-surjectivity} contain the proofs of Theorems \ref{thm:only-if} and \ref{thm:surjectivity} respectively.

\section{Background and definitions} \label{sec:background}

This section consists of background and definitions relating to root systems, Weyl groups, and the weak and strong Bruhat orders; all of this material is standard and may be found, for example, in \cite{Bjorner-Brenti}.

Throughout the paper, $\Phi$ will denote a finite, crystallographic root system with chosen set of simple roots $\Delta$ and corresponding set of positive roots $\Phi^+$.  We freely use the well-known Cartan-Killing classification of irreducible root systems into types $A_n, B_n, C_n, D_n, G_2, F_4, E_6, E_7$, and $E_8$, although all of our results hold (at least conjecturally) in arbitrary type, and many of our proofs are type-uniform.

The \emph{root poset} is the partial order $(\Phi^+, \leq)$ where $\beta \leq \beta'$ if $\beta'-\beta$ is a nonnegative sum of simple roots.

We write $s_{\alpha}$ for the simple reflection across the hyperplane orthogonal to the simple root $\alpha \in \Delta$, and $W(\Phi)$ for the Weyl group, which is generated by the simple reflections.  Given an element $w\in W(\Phi)$, its \emph{length} $\ell(w)$ is defined to be the smallest $\ell$ such that $w=s_{\alpha_1} \cdots s_{\alpha_{\ell}}$ for some sequence of simple reflections.  The \emph{inversion set} of $w$ is: 
\[
I_{\Phi}(w) = \{ \beta \in \Phi^+ \: | \: w\beta \in \Phi^-\}.
\]
It is well-known that $\ell(w)=|I_{\Phi}(w)|$ and that $W(\Phi)$ has a unique element $w_0$ of maximal length; $w_0$ is an involution and has $I_{\Phi}(w_0)=\Phi^+$.  Inversions $\beta$ which are simple roots are called \emph{descents}.  We also say a simple reflection $s_{\beta}$ is a \emph{right (resp. left) descent} of $w$ if $\ell(ws_{\beta})<\ell(w)$ (resp. $\ell(s_{\beta}w)<\ell(w)$; the simple reflection $s_{\beta}$ is a right descent if and only if the simple root $\beta$ is a descent.

\begin{prop} \label{prop:biconvex}
Elements $w \in W(\Phi)$ are uniquely determined by their inversion sets, and $S \subseteq \Phi^+$ is the inversion set of some element if and only if it is \emph{biconvex}:
\begin{itemize}
    \item For each pair $\alpha, \beta \in S$, if $\alpha + \beta \in \Phi^+$, then $\alpha + \beta \in S$, and
    \item If $\gamma \in S$ and $\gamma = \alpha + \beta$ with $\alpha, \beta \in \Phi^+$, then at least one of $\alpha, \beta$ must be in $S$.
\end{itemize}
\end{prop}

The \emph{left weak order} (sometimes called the \emph{left weak Bruhat order}) on $W(\Phi)$ is determined by its cover relations: $w \lessdot_L s_{\alpha}w$ whenever $\ell(s_{\alpha}w)=\ell(w)+1$, where $\alpha \in \Delta$.  The \emph{right weak order} is defined analogously, except with right multiplication by $s_{\alpha}$.  All Weyl groups are assumed to be ordered by left weak order unless otherwise specified.  It is a nontrivial fact that the weak orders are lattices.  We denote the lattice operations of join and meet by $\lor$ and $\land$ respectively, with superscripts $L$ or $R$ to indicate either the left or right weak order.

\begin{prop} \label{prop:weak-given-by-inversions}
The left weak order on $W(\Phi)$ is given by containment of inversion sets, that is: $u \leq_L w$ if and only if $I_{\Phi}(u) \subseteq I_{\Phi}(w)$.
\end{prop}

The map $w \mapsto w^{-1}$ defines a poset isomorphism between the left and right weak orders.  Each has a unique minimal element $e$, the group identity element, and $w_0$ as its unique maximal element, called the \emph{longest element}.  Both left and right multiplication by $w_0$ determine poset anti-automorphisms of both left and right weak order.  We note that $I_{\Phi}(w_0w)=\Phi^+ \setminus I_{\Phi}(w)$.

If $W=W(\Phi)$ is a Weyl group with simple roots $\Delta$, and $J \subseteq \Delta$, we let $W_J$ denote the \emph{parabolic subgroup} of $W$ generated by $\{s_{\alpha}\}_{\alpha \in J}$.  The \emph{parabolic quotient} $W^J$ is the set of elements of $W$ with no descents in $J$.  We let $\Phi_J$ be the root system of those roots in $\Phi$ which are linear combinations of elements of $J$.

\begin{prop} \label{prop:parabolic-quotient}
Let $W=W(\Phi)$ and let $J \subseteq \Delta$, then:
\begin{itemize}
    \item $W^J$ forms a system of coset representatives for $W_J$ in $W$; in particular, each $w \in W$ has a unique expression $w=w^J w_J$ with $w^J \in W^J$ and $w_J \in W_J$.  For each $J$, by taking $w=w_0$, this expression determines important elements $w_0^J$ and $w_{0J}$.  We write $w_0(J)$ for $w_{0J}$ to avoid excessive subscripts.
    \item $W^J=[e,w_0^J]_L$ and $W_J=[e,w_0(J)]_L=[e,w_0(J)]_R$.
    \item The elements of $W^J$ are the unique elements of minimal length in their $W_J$-cosets, and the above expression for $w$ is \emph{length-additive}: $\ell(w)=\ell(w^J)+\ell(w_J)$.
\end{itemize}
\end{prop}

Let $\Phi$ be a root system with positive roots $\Phi^+$.  A subset $\Phi' \subset \Phi$ is a \emph{subsystem} of $\Phi$ if $\Phi'=\Phi \cap U$ for some linear subspace $U$ of $\spn(\Phi)$.  It is clear that any such $\Phi'$ is itself a root system.  The following generalization of pattern avoidance to finite Weyl groups was introduced by Billey and Postnikov \cite{Billey2005Smoothness}.  For $w \in W(\Phi)$, we say $w$ \emph{contains the pattern} $(w',\Phi')$ if $I_{\Phi}(w)\cap U = I_{\Phi'}(w')$; we write $w|_{\Phi'}=w'$ in this case.  If $\Phi'=\Phi_J$ is the set of roots in the span of $J \subseteq \Delta$ then $w|_{\Phi'}=w_J$, if we identify $W(\Phi')$ with $W_J$ in the natural way; note, however, that many subsystems are not of this form.  We say $w$ \emph{avoids} $(w',\Phi')$ if it does not contain any pattern isomorphic to $(w',\Phi')$.

A ranked poset $P=P_0 \sqcup P_1 \sqcup \cdots \sqcup P_r$ (such as the left or right weak order on a Weyl group, which are ranked by length) is \emph{rank-symmetric} if $|P_i|=|P_{r-i}|$ for all $i$, and \emph{rank-unimodal} if $|P_0|\leq \cdots \leq |P_j| \geq \cdots \geq |P_r|$ for some $j$.  Its \emph{rank generating function} $P(q)$ is $\sum_{i=0}^r |P_i| q^i$.  It is well known that $W_J$ and $W^J$ are rank-symmetric and rank-unimodal for all $J \subseteq \Delta$, and Proposition \ref{prop:parabolic-quotient} implies that 
\begin{equation} \label{eq:parabolic-product}
W^J(q)W_J(q)=W(q).
\end{equation}
We let $\Lambda^L_w=[e,w]_L$ and $V^L_w=[w,w_0]_L$ denote the principal lower and upper order ideals in left weak order respectively, and similarly for right weak order; we sometimes suppress the decorations $L$ or $R$ if a claim works just as well in either left or right weak order.  We make the convention that the rank function on $V_w$ is the natural one viewing $V_w$ as a poset in its own right: an element $u$ of $V_w$ has rank $\ell(u)-\ell(w)$.

\begin{prop} \label{prop:product-of-degrees}
For any finite Weyl group $W$ of rank $n$ we have 
\[
W(q)=\prod_{i=1}^n [d_i]_q,
\]
where the $d_i$ are integer invariants called the \emph{degrees} of $W$, and $[d]_q$ denotes the $q$-integer $1+q+\cdots+q^{d-1}$. 
\end{prop}

In addition to $q$-integers $[d]_q$, we will need $q$-factorials $[d]_q ! := [1]_q [2]_q \cdots [d]_q$, and $q$-multinomial coefficients
\[
\qbinom{\sum_{i=1}^k a_i}{a_1,\ldots,a_k}_q := \frac{[\sum_i a_i]_q!}{[a_1]_q! \cdots [a_k]_q!},
\]
with the usual convention that for binomial coefficients only $a_1$ is written on the bottom.  For a polynomial $f(q)$, we write $[q^d]f$ to denote the coefficient of $q^d$ in $f$.  

We will also use a third partial order, the \emph{strong Bruhat order} (also called just the \emph{strong order} or \emph{Bruhat order}), on the elements of a finite Weyl group $W$.  The cover relations in Bruhat order are $w \lessdot_B wt$ whenever $\ell(wt)=\ell(w)+1$, where $t$ is any reflection in $W$ (not necessarily a simple reflection).  Note that the Bruhat order has the same rank structure as the left and right weak orders, but strictly more cover relations.

\section{Separable elements of Weyl groups} \label{sec:separable}

We now introduce a definition of a separable element in any finite Weyl group.  This definition coincides exactly with separable permutations in the case of the symmetric group, although this is only made clear by Theorem \ref{thm:pattern-classification} below, where separable elements are characterized by root system pattern avoidance.  Theorem \ref{thm:bijection-to-faces} in Section \ref{sec:associahedra} gives another characterization of separable elements.

\begin{defin} \label{def:separable}
Let $w\in W(\Phi)$. Then $w$ is \textit{separable} if one of the following holds:
\begin{enumerate}
\item[(S1)] $\Phi$ is of type $A_1$;
\item[(S2)] $\Phi=\bigoplus\Phi_i$ is reducible and $w|_{\Phi_i}$ is separable for each $i$;
\item[(S3)] $\Phi$ is irreducible and there exists a \textit{pivot} $\alpha_i\in\Delta$ such that $w|_{\Phi_J}\in W(\Phi_J)$ is separable where $\Phi_J$ is generated by $J=\Delta\setminus\{\alpha_i\}$ and such that either 
\begin{align*}
    & \{\beta\in\Phi^+:\beta\geq\alpha_i\}\subset I_{\Phi}(w), \text{ or} \\
    & \{\beta\in\Phi^+:\beta\geq\alpha_i\}\cap I_{\Phi}(w)=\emptyset.
\end{align*}
We say $\alpha_i$ is a \emph{full} pivot in the first case, and an \emph{empty} pivot in the second.
\end{enumerate}

This notion is well-defined, since, in (S2) and (S3), we reduce to a subsystem of strictly smaller rank.
\end{defin}

\begin{example} \label{ex:sep-in-B4}
Let $\Phi=\{\pm e_i \pm e_j \: | \: 1 \leq i < j \leq 4\} \sqcup \{\pm e_i \: | \: 1 \leq i \leq 4\}$ be the root system of type $B_4$, where the $e_i$ are the standard basis elements in $\mathbb{R}^4$; let $\alpha_1=e_1-e_2$, $\alpha_2=e_2-e_3$, $\alpha_3=e_3-e_4$, and $\alpha_4=e_4$ denote the simple roots.  Let $w \in W(\Phi)$ be the element whose inversion set $I_{\Phi}(w) \subseteq \Phi^+$ is indicated in Figure \ref{fig:B4}.  Then we can conclude $w$ is separable as follows:
\begin{itemize}
    \item First, by (S3), we see that $\alpha_3$ is a full pivot since all $\beta \geq \alpha_3$ are in the inversion set.
    \item Now we reduce to checking that $w|_{\Phi_J}$ is separable, with $J=\Delta \setminus \{\alpha_3\}=\{\alpha_1, \alpha_2, \alpha_4\}$.  Notice $\Phi_J = \Psi_1 \oplus \Psi_2$ is reducible, with $\Psi_1^+=\{\alpha_1, \alpha_2, \alpha_1+\alpha_2\}$ and $\Psi_2^+=\{\alpha_4\}$, so by (S2), we need to show that $w|_{\Psi_1}$ and $w|_{\Psi_2}$ are separable.  Since $\Psi_2$ is of type $A_1$, $w|_{\Psi_2}$ is separable by (S1) of the definition.  
    \item Finally, $w|_{\Psi_1}$ has a pivot $\alpha_1$, this is an empty pivot, since neither $\alpha_1$ nor $\alpha_1+\alpha_2$ is an inversion.  We then reduce to the type $A_1$ subsystem spanned by $\alpha_2$, and we are done by (S1).
\end{itemize}
\end{example}

\begin{figure}[ht]
    \begin{center}
    \begin{tikzpicture}[scale=0.8]
    \node[draw,shape=circle,fill=white,scale=0.5](a)[label=below: {$\alpha_1$}] at (-4,0) {};
    \node[draw,shape=circle,fill=black,scale=0.5](b)[label=below: {$\alpha_2$}] at (-2,0) {};
    \node[draw,shape=circle,fill=black,scale=0.5](c)[label=below: {$\alpha_3$}] at (0,0) {};
    \node[draw,shape=circle,fill=white,scale=0.5](d)[label=below: {$\alpha_4$}] at (2,0) {};
    
    \node[draw,shape=circle,fill=white,scale=0.5](e)[label=left: {$\alpha_1+\alpha_2$}] at (-3,1) {};
    \node[draw,shape=circle,fill=black,scale=0.5](f) at (-1,1) {};
    \node[draw,shape=circle,fill=black,scale=0.5](g) at (1,1) {};
    
    \node[draw,shape=circle,fill=black,scale=0.5](h) at (-2,2) {};
    \node[draw,shape=circle,fill=black,scale=0.5](i) at (0,2) {};
    \node[draw,shape=circle,fill=black,scale=0.5](j) at (2,2) {};
    
    \node[draw,shape=circle,fill=black,scale=0.5](k) at (-1,3) {};
    \node[draw,shape=circle,fill=black,scale=0.5](l) at (1,3) {};
    
    \node[draw,shape=circle,fill=black,scale=0.5](m) at (0,4) {};
    \node[draw,shape=circle,fill=black,scale=0.5](n) at (2,4) {};
    
    \node[draw,shape=circle,fill=black,scale=0.5](o) at (1,5) {};
    
    \node[draw,shape=circle,fill=black,scale=0.5](p) at (2,6) {};
    
    \draw (a)--(e)--(b)--(f)--(c)--(g)--(d);
    \draw (e)--(h)--(f)--(i)--(g)--(j);
    \draw (h)--(k)--(i)--(l)--(j);
    \draw (k)--(m)--(l)--(n);
    \draw (m)--(o)--(n);
    \draw (o)--(p);
    
    
    \end{tikzpicture}
    \caption{The root poset for type $B_4$ is shown, with the filled nodes indicating the positive roots in the inversion set of the element $w$ from Example \ref{ex:sep-in-B4}.}
    \label{fig:B4}
    \end{center}
\end{figure}
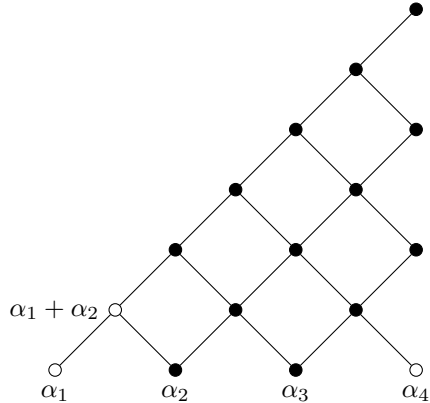

Theorem \ref{thm:symmetric-and-unimodal} from \cite{First-separable-paper} generalizes Fan Wei's result for the symmetric group to general finite Weyl groups.

\begin{thm}[\cite{First-separable-paper}] \label{thm:symmetric-and-unimodal}
Let $w \in W$ be separable, then $\Lambda_w$ and $V_w$ are rank-symmetric and rank-unimodal, and
\begin{equation} \label{eq:factorization}
\Lambda_w(q)V_w(q)=W(q).
\end{equation}
\end{thm}

The similarity of (\ref{eq:factorization}) to (\ref{eq:parabolic-product}) suggests that one should look for a length-additive multiplicative decomposition of $W$ corresponding to each separable element $w$, analogous to that in Proposition \ref{prop:parabolic-quotient}.  Indeed, such a decomposition is constructed in Section \ref{sec:gen-quotients}; in addition we show for the symmetric group (and conjecture in other types) that separable elements induce \emph{all} such decompositions.

In Section \ref{sec:associahedra} we give explicit product formulas for $\Lambda_w(q)$ and $V_w(q)$ when $w$ is separable in terms of the nested set indexing the corresponding face of the graph associahedron, making Theorem \ref{thm:symmetric-and-unimodal} even more explicit.

Recall that separable permutations are defined to be those which avoid the patterns 3142 and 2413.  Theorem \ref{thm:pattern-classification} from \cite{First-separable-paper} implies that separable elements of general finite Weyl groups are characterized by pattern avoidance in the sense of Billey and Postnikov (see the discussion after Proposition \ref{prop:parabolic-quotient} in Section \ref{sec:background}).  This has the benefit of giving a non-recursive characterization of separable elements (in contrast to Definition \ref{def:separable}) as well as implying that separable elements in $W=S_n$ are precisely the separable permutations which had received much previous study.

\begin{thm}[\cite{First-separable-paper}] \label{thm:pattern-classification}
An element $w \in W(\Phi)$ is separable if and only if $w$ avoids the following root system patterns:
\begin{itemize}
    \item [(i)] the patterns corresponding to the permutations 3142 and 2413 in the Weyl group of type $A_3$,
    \item [(ii)] the two patterns of length two in the Weyl group of type $B_2$, and
    \item [(iii)] the six patterns of lengths two, three, and four in the Weyl group of type $G_2$.
\end{itemize}
\end{thm}

\begin{cor}[\cite{First-separable-paper}]
Under the usual identification of the Weyl group $W$ of type $A_{n-1}$ with the symmetric group of permutations of $\{1,\ldots,n\}$, an element $w \in W$ is separable if and only if it corresponds to a separable permutation.
\end{cor}

Theorem \ref{thm:pattern-classification} also makes it clear that the set of separable elements is closed under the natural involutions on Weyl groups $x \mapsto w_0x$, $x \mapsto xw_0$, and $x \mapsto x^{-1}$; the latter two of these are not clear from Definition \ref{def:separable}.

\begin{cor}[\cite{First-separable-paper}] \label{cor:closure}
Let $w \in W$ be separable.  Then $w_0w, ww_0,$ and $w^{-1}$ are also separable.
\end{cor}
\begin{proof}
The set of forbidden patterns in Theorem \ref{thm:pattern-classification} is closed under these three involutions, and it is easy to check that $w$ avoids $u$ if and only if $w^{-1}$ avoids $u^{-1}$ (and similarly for the other two).
\end{proof}

\begin{prop} \label{prop:separable-in-direct-sum}
An element $w=(w_1,w_2) \in W(\Phi_1) \times W(\Phi_2) = W(\Phi_1 \oplus \Phi_2)$ is separable if and only if $w_1 \in W(\Phi_1)$ and $w_2 \in W(\Phi_2)$ are both separable.
\end{prop}
\begin{proof}
This is clear from Definition \ref{def:separable} after noting that $I_{\Phi_1 \oplus \Phi_2}(w)=I_{\Phi_1}(w_1) \sqcup I_{\Phi_2}(w_2)$ and that the root poset on $\Phi_1 \oplus \Phi_2$ is just the disjoint union of the root posets of $\Phi_1$ and $\Phi_2$.
\end{proof}

\section{Generalized quotients and splittings of Weyl groups} \label{sec:gen-quotients}
Given any subset $U$ of a Weyl group $W$, Bj\"{o}rner and Wachs \cite{Bjorner-Wachs} introduced the \emph{generalized quotient}:
\[
W/U=\{w \in W \: | \: \ell(wu)=\ell(w)+\ell(u), \forall u \in U\}.
\]

\begin{prop}[Bj\"{o}rner and Wachs \cite{Bjorner-Wachs}]\label{prop:W/U}
Let $u_0=\bigvee^R_{u \in U} u$, then $W/U=[e,w_0u_0^{-1}]_L$.
\end{prop}

A pair $(X,Y)$ of arbitrary subsets $X,Y \subseteq W$ such that the multiplication map $X \times Y \to W$ sending $(x,y)\mapsto xy$ is \emph{length-additive} (meaning $\ell(xy)=\ell(x)+\ell(y), \forall x \in X, y \in Y$) and bijective is called a \emph{splitting} of $W$.  Generalized quotients generalize the notion of parabolic quotients, since $W^J=W/W_J$; Proposition \ref{prop:parabolic-quotient} implies that we have a splitting $(W^J, W_J)$ in this case.

Bj\"{o}rner and Wachs (1988) asked for a classification of splittings of the symmetric group induced by generalized quotients:

\begin{prob}[Bj\"{o}rner and Wachs \cite{Bjorner-Wachs}] \label{prob:splittings}
In the case $W=S_n$, for which $U \subseteq W$ is the multiplication map
\[
W/U \times U \to W
\]
sending $(x,y) \mapsto xy$ a splitting?
\end{prob}

Since this map is length-additive by definition of generalized quotient, Problem \ref{prob:splittings} amounts to asking when it is a bijection.  Theorem \ref{thm:separable-splitting} identifies splittings corresponding to separable elements in any finite Weyl group.  Fan Wei  \cite{Wei2012Product} proved an equivalent statement in the case of the symmetric group using explicit manipulations on permutations; our proof is type-independent.  Theorem \ref{thm:separable-splitting} answers an open problem of Wei by extending the results of \cite{Wei2012Product} to all finite Weyl groups.

\begin{thm} \label{thm:separable-splitting}
Let $W$ be any finite Weyl group and $U=[e,u]_R$ with $u$ separable, then $(W/U, U)$ is a splitting of $W$. 
\end{thm}
\begin{proof}
By Corollary \ref{cor:closure}, the set of separable elements is closed under the involutions of multiplying on either side by $w_0$ and inversion.  By Proposition \ref{prop:W/U}, we have $W/U=[e,w_0u^{-1}]_L$, and by applying some symmetries of weak order we have
\begin{align*}
    [e,u]_R&=[uw_0,w_0]_R \cdot w_0 \\
    &=[w_0u^{-1},w_0]_L^{-1} \cdot w_0.
\end{align*}
Thus to verify that the multiplication map $W/U \times U \to W$ is a bijection (it is automatically length-additive) it suffices to prove that the map $\Lambda^L_{\pi} \times V^L_{\pi} \to W$ given by $(x,y) \mapsto xy^{-1}w_0$ is bijective for the separable element $\pi=w_0u^{-1}$.  Since taking inverses and multiplying by $w_0$ are involutions on $W$, this in turn is equivalent to checking that $(x,y) \mapsto yx^{-1}$ is a bijection, which we now do.  In light of Theorem \ref{thm:symmetric-and-unimodal}, it suffices to prove surjectivity, so fix $w \in W$ which we will show is in the image of this map.  Assume without loss of generality that $W=W(\Phi)$ is irreducible and $\pi$ is separable with an empty pivot $\alpha_i$, the other case in (S3) being analogous.  Let $J=\Delta \setminus \{\alpha_i\}$ and let $\Phi'$ be the parabolic subsystem generated by $J$.

By induction on rank, we may assume that the claim is true for $W(\Phi')$, so there exist elements $x' \in \Lambda_{\pi'}$ and $y' \in V_{\pi'}$ such that $y'(x')^{-1}=w'$, where $\pi'=\pi|_{\Phi'}$ and $w'=w|_{\Phi'}$.  The element $\pi'$ is still separable, since Theorem \ref{thm:pattern-classification} implies that patterns of separable elements are separable.  Now, viewing $w' \in W(\Phi') \subset W(\Phi)$ as an element of the full group, we have that $w \geq_L w'$, by comparing inversion sets and applying Proposition \ref{prop:weak-given-by-inversions}.  This means that we can write $w=s_{i_1} \cdots s_{i_k}w'$ with lengths adding.  In fact, we have that $w'=w_J$ and $s_{i_1} \cdots s_{i_k} = w^J \in W^J$.  In particular, since $y' \in W_{J}$ we know that $s_{i_1}\cdots s_{i_k}y'$ is reduced; call this element $y$, so $y \geq_L y' \geq_L \pi'=\pi$, thus $y \in V^L_{\pi}$.  We have
\[
y(x')^{-1}=w(w')^{-1}y'(x')^{-1}=w
\]
as desired.
\end{proof}

In Theorem \ref{thm:only-if} we answer Problem \ref{prob:splittings}; in fact we show more, by ruling out splittings not coming from a generalized quotient.  The proof of Theorem \ref{thm:only-if} appears in Section \ref{sec:proof-of-only-if}.

\begin{thm} \label{thm:only-if}
Let $(X,Y)$ be an arbitrary splitting of $W=S_n$, then $X=W/Y$ and $Y=[e,u]_R$ with $u$ separable.
\end{thm}

Theorems \ref{thm:separable-splitting} and \ref{thm:only-if} show that generalized quotients with $U=[e,u]_R$ and $u$ separable are exactly those for which the multiplication map $W/U \times U \to W=S_n$ is a bijection.  Theorem \ref{thm:surjectivity} shows that this map is a \emph{surjection} for every $u$.  

\begin{thm} \label{thm:surjectivity}
Let $u$ be any element of $W=S_n$ and $U=[e,u]_R$, then the multiplication map $W/U \times U \to W$ is surjective.
\end{thm}

Despite its simple statement, Theorem \ref{thm:surjectivity} is surprisingly difficult to prove, and requires exploiting new connections between the left and right weak orders and the strong Bruhat order; the proof of Theorem \ref{thm:surjectivity} appears in Section \ref{sec:proof-of-surjectivity}.  As an indication of the strength of this Theorem, we discuss in Section \ref{sec:linear-extensions} how it immediately solves an open problem of Morales, Pak, and Panova about linear extensions of posets.

\subsection{Linear extensions and weak order} \label{sec:linear-extensions}

In this section we sketch how Theorem \ref{thm:surjectivity} resolves an open problem of Pak, Panova, and Morales \cite{PPM} by giving a direct combinatorial proof of an inequality due to Sidorenko (Theorem \ref{thm:sidorenko}).  We also give a new $q$-analog of Sidorenko's inequality in Corollary \ref{cor:q-analog-of-sidorenko}.

See \cite{BW-lin-ext} for the following background on linear extensions.  A \emph{linear extension} of a finite poset $P=(\{p_1,\ldots ,p_n\}, \leq_P)$ is an order preserving bijection $\lambda: P \to [n]$, where $[n]$ denotes the set $\{1,\ldots,n\}$ under the usual ordering.  We write $e(P)$ for the number of linear extensions of $P$.  The \emph{order dimension} of $P$ is the smallest number $t$ such that there exist linear extensions $\lambda_1,\ldots,\lambda_t$ such that for all $i,j$ we have $p_i \leq_P p_j$ if and only if $\lambda_k(p_i) \leq \lambda_k(p_j)$ for all $k=1,\ldots,t$.  In this situation we write $P=\bigcap_{k=1}^t \lambda_k$.

We say $P$ is \emph{naturally labelled} if $p_i \mapsto i, \forall i$ is a linear extension.  We may identify linear extensions $\lambda$ of $P$ with permutations in $S_n$ by identifying the linear extension $p_i \mapsto \pi_i, \forall i$ with the permutation $\pi=\pi_1 \ldots \pi_n$, in this case we write $\lambda_{\pi}$ for $\lambda$.  For $\pi \in S_n$, write $P_{\pi}$ for the poset on $\{p_1,...,p_n\}$ defined by $P_{\pi}=\lambda_e \cap \lambda_{\pi}$, such a poset is always naturally labelled.  Two-dimensional posets $P$ have natural \emph{complementary posets} $\overline{P}$ defined as follows: choose an isomorphism from $P$ to some $P_{\pi}$ (this can always be done), and let $\overline{P}=P_{\pi w_0}$.  The poset $\overline{P}$ may not be uniquely determined, as there may be multiple choices for $\pi$, however Theorem \ref{thm:sidorenko} holds for any complement formed from this construction.

A poset $P$ is \emph{series-parallel} if can be formed from combining some number of singleton posets using the operations of disjoint union (elements of $Q$ are incomparable with elements of $Q'$ in $Q \sqcup Q'$) and direct sum (all elements of $Q$ are less than all elements of $Q'$ in $Q \oplus Q'$).

\begin{thm}[Sidorenko \cite{Sidorenko}] \label{thm:sidorenko}
Let $P$ be a two-dimensional poset, then:
\[
e(P)e(\overline{P}) \geq n!,
\]
with equality if and only if $P$ is series-parallel.
\end{thm}

Sidorenko's original proof of Theorem \ref{thm:sidorenko} uses intricate analysis of various recurrences and the Max-flow/Min-cut Theorem.  It was reproven by Bollob\'{a}s, Brightwell, and Sidorenko \cite{BBS} using a known special case of the still-open Mahler conjecture from convex geometry and an implication of the difficult Perfect Graph Theorem.  This led Pak, Panova, and Morales \cite{PPM} to state an open problem asking for a direct combinatorial proof; we provide such a proof by applying Theorem \ref{thm:surjectivity}.

\begin{prop}[Bj\"{o}rner and Wachs \cite{BW-lin-ext}] \label{prop:le-from-intervals}
The linear extensions of $P_{\pi}$ are exactly $\{\lambda_u \: | \: u \in [e,\pi]_R\}$.
\end{prop}

\begin{proof}[New proof of Theorem \ref{thm:sidorenko}]
Pick $\pi$ such that $P$ is isomorphic to $P_{\pi}$.  By Proposition \ref{prop:le-from-intervals}, we need to show that $|[e,\pi]_R| \cdot |[e,\pi w_0]|_R \geq n!$.  We simply observe that inversion gives a bijection $[e,\pi w_0]_R \to [e,w_0 \pi^{-1}]_L=W/[e,\pi]_R$, and apply Theorem \ref{thm:surjectivity}.  Thus we have a simply-defined (just group multiplication) surjection from the set $W/[e,\pi]_R \times [e,\pi]_R$ of cardinality $e(P)e(\overline{P})$ to the set $W=S_n$ of cardinality $n!$. To get the equality case, note that Theorems \ref{thm:separable-splitting}, \ref{thm:only-if}, and \ref{thm:surjectivity} together imply that we have equality if and only if $\pi$ is separable.  It is easy to check that the two cases in (S3) correspond to the operations $\oplus$ and $\sqcup$ on posets, so that $\pi$ is separable if and only if $P_{\pi}$ is series-parallel.
\end{proof}

Let 
\[
e(P_{\pi};q)=[e,\pi]_R(q)=\sum_{u \in [e,\pi]_R} q^{\ell(u)}
\]
denote the generating function for linear extensions of $P_{\pi}$ graded by length (see Proposition \ref{prop:le-from-intervals}).  The following Corollary is a new $q$-analog of Theorem \ref{thm:sidorenko}.

\begin{cor} \label{cor:q-analog-of-sidorenko}
For any $\pi \in S_n$ and any $d \geq 0$ we have:
\[
[q^d]\left(e(P_{\pi};q)e(P_{\pi w_0};q)\right) \geq [q^d] [n]_q!.
\]
\end{cor}
\begin{proof}
Since the surjection in Theorem \ref{thm:surjectivity} is length-additive, this corollary follows from the same argument given for our new proof of Theorem \ref{thm:sidorenko}.
\end{proof}

\subsection{Splittings and surjectivity in other Weyl groups} \label{sec:conjectures}

Although Theorem \ref{thm:separable-splitting} holds for all finite Weyl groups, Theorems \ref{thm:only-if} and \ref{thm:surjectivity} are currently stated only for $W=S_n$.  We conjecture that both extend to arbitrary finite Weyl groups, with an additional restriction in Theorem \ref{thm:only-if}.

\begin{conj}
Theorem \ref{thm:surjectivity} holds for any finite Weyl group $W$.
\end{conj}

\begin{conj} \label{conj:other-types-only-if}
Let $[e,u]_R=U \subseteq W$, then $(W/U,U)$ is a splitting of $W$ if and only if $u$ is separable.
\end{conj}

\begin{remark}
For the exceptional Weyl group $W$ of type $F_4$, there is a splitting $(W/U,U)$ where $U$ is not an interval in right weak order; this is why the statement of Conjecture \ref{conj:other-types-only-if} is weaker than that of Theorem \ref{thm:only-if}.  It may be that the full strength of Theorem \ref{thm:only-if} holds for the remaining infinite families of Weyl groups of types $B_n=C_n$ and $D_n$.
\end{remark}

\section{Product formulas and graph associahedra} \label{sec:associahedra}

In this section we show that separable elements in $W$ are in bijection with the faces of all dimensions of $2^r$ copies of the \emph{graph associahedron} $A(\Gamma)$ of the Dynkin diagram $\Gamma$ for $W$, where $W$ has $r$ irreducible factors.  The Dynkin diagram is a graph with vertices indexed by the simple roots $\Delta$ and edges $\overline{\alpha \alpha'}$ whenever $s_{\alpha}$ and $s_{\alpha'}$ do not commute; we often identify subgraphs of $\Gamma$ with the corresponding subsets of $\Delta$ when convenient.  It is well-known that all connected components of Dynkin diagrams of finite Weyl groups are trees.  Much useful information about a separable element $w$, such as its Lehmer code, the rank generating functions $\Lambda_w(q)$ and $V_w(q)$, and a factorization of $w$ as a product of elements of the form $w_0(J)$ can be read off from the corresponding face of $A(\Gamma)$.

Given a graph $\Gamma$, the \emph{graph associahedron} $A(\Gamma)$ is a convex polytope which can be defined as the Minkowski sum of coordinate simplices corresponding to the connected subgraphs of $\Gamma$.  First arising in the work of De Concini and Procesi on wonderful models of subspace arrangements \cite{DeConcini-Procesi}, these polytopes have received intensive study, especially in the case when $\Gamma$ is a Dynkin diagram.  When $\Gamma$ is the Dynkin diagram of type $A_n$, a line graph, $A(\Gamma)$ is the usual Stasheff Associahedron.  


We will use a model for the faces of $A(\Gamma)$ due originally to Carr and Devadoss \cite{Carr-Devadoss}; our treatment follows the presentation of Postnikov \cite{beyond-paper}. A collection $\mathcal{N}$ of subsets of $\Gamma$ is a \emph{nested set} if:
\begin{itemize}
    \item[(N1)] For all $J \in \mathcal{N}$, the induced subgraph $\Gamma|_J$ on the vertex set $J$ is connected.
    \item[(N2)] For any $I,J \in \mathcal{N}$ we have either $I \subseteq J,$ or $J \subseteq I,$ or $I \cap J = \emptyset$.
    \item[(N3)] For any collection of $k \geq 2$ disjoint subsets $J_1,...,J_k \in \mathcal{N}$, the subgraph $\Gamma|_{J_1 \cup \cdots \cup J_k}$ is \emph{not} connected.
\end{itemize}
The relevant notion of connectivity for directed graphs is the connectivity of the associated simple undirected graph, so the structure of $A(\Gamma)$ does not depend on an orientation of $\Gamma$.  This is why we have omitted reference to the edge multiplicities and orientations in our definition of Dynkin diagrams.

\begin{prop}[Carr and Devadoss \cite{Carr-Devadoss}; Postnikov \cite{beyond-paper}] \label{prop:faces-are-nested}
The poset of faces of $A(\Gamma)$ is isomorphic to the poset of nested sets on $\Gamma$ which contain all connected components of $\Gamma$, ordered by reverse containment.
\end{prop}

We call a total ordering of the elements of a nested set $\mathcal{N}$ \emph{monotonic} if $J$ appears after $I$ whenever $J \subseteq I$.  The \emph{depth} of $J \in \mathcal{N}$ is the maximum length $k$ of a chain $J \subsetneq I_1 \subsetneq I_2 \subsetneq \cdots \subsetneq I_k$ of elements of $\mathcal{N}$.  We let $\mathcal{N}_{even}$ and $\mathcal{N}_{odd}$ denote the elements of $\mathcal{N}$ of even and odd depth respectively.

\begin{thm} \label{thm:bijection-to-faces}
Let $W$ be a finite Weyl group with $r$ irreducible components and with Dynkin diagram $\Gamma$, then:
\begin{enumerate}
    \item The nested sets on $\Gamma$ are in bijection with the separable elements of $W$ via the map 
    \[
    \mathcal{N} \mapsto \prod_{J \in \mathcal{N}} w_0(J):=w(\mathcal{N}), 
    \]
    where the product is taken in any monotonic order.  Thus the number of separable elements of $W$ is $2^r$ times the number of faces of $A(\Gamma)$.
    \item The weak order rank generating functions of the intervals $[e, w(\mathcal{N})]$ in left and right order are:
    \[
    \Lambda^L_{w(\mathcal{N})}(q)=q^{\ell(w(\mathcal{N}))}\Lambda^R_{w(\mathcal{N})}(q^{-1})=\frac{\prod_{J \in \mathcal{N}_{even}} W_J(q)}{\prod_{J \in \mathcal{N}_{odd}} W_J(q)}.
    \]
\end{enumerate}
\end{thm}

The proof of Theorem \ref{thm:bijection-to-faces} appears in Section \ref{sec:proof-of-bijection} below.

\begin{remark}
By Proposition \ref{prop:product-of-degrees}, the rank generating functions $W_J(q)$ appearing in the product formula are themselves products of the $q$-integers of the degrees of $W_J$, thus one may expand the formula in Theorem \ref{thm:bijection-to-faces} (2) as a quotient of products of $q$-integers.  Note that the $W_J$ appearing in this formula may be of several different Cartan-Killing types and so this product will contain degrees from these several families.  The $q$-integers in the denominator do not always pair up to cancel with those in the numerator, so even the fact that this quotient is a polynomial is nontrivial.

This product formula generalizes several known formulas for rank generating functions of intervals in the weak order.  In the case of 132-avoiding permutations (a subset of separable permutations) this formula is equivalent to the hook-length formula for linear extensions of trees (see \cite{BW-lin-ext}).  More generally, for separable permutations Wei \cite{Wei2012Product} gives a formula in terms of \emph{separating trees} which is equivalent to (2) in the case $W=S_n$, however Wei's formula has several cases and the separating tree corresponding to $w$ is not uniquely defined.  It is known that computing even the size of weak order intervals is $\# P$-hard \cite{Dittmer-Pak}, so there can be no nice formulas for $\Lambda_w(q)$ in general, making this formula notable. 
\end{remark}

We now describe the inverse to the map $\mc{N} \mapsto w(\mc{N})$; that this map is indeed an inverse is proven in Section \ref{sec:proof-of-bijection}.  Given a separable element $w$ of $W$, we inductively construct a nested set $\mc{N}_w$ as follows.  If $W$ is of type $A_1$, so that $\Gamma=\{\alpha_1\}$ is a single vertex, we set $\mc{N}_e=\emptyset$ and $\mc{N}_{\tau}=\{\Gamma\}$, where $\tau$ is the non-identity element of $W$.  If $W=W_1 \times W_2$ is reducible, and $w=(w_1,w_2)$, then $\mc{N}_w=\mc{N}_{w_1} \sqcup \mc{N}_{w_2}$.  Otherwise, assume we have constructed $\mc{N}_u$ for separable elements $u$ in all Weyl groups of smaller rank than $W$, and that $W$ is irreducible.  Let $\alpha \in \Delta$ be a pivot of $w$; let $\Gamma_1,...,\Gamma_k$ be the connected components of $\Gamma \setminus \{\alpha\}$.
\begin{itemize}
    \item If $\alpha \in \Delta$ is an empty pivot of $w$, then let \[
    \mc{N}_w = \bigsqcup_{i=1}^k \mc{N}_{w_{\Gamma_i}}
    \]
    \item If $\alpha \in \Delta$ is a full pivot of $w$, then $w_0w$ is separable and has $\alpha$ as an empty pivot; in this case we let
    \[
    \mc{N}_w = \{\Gamma\} \sqcup \mc{N}_{w_0w}.
    \]
\end{itemize}
That $\mc{N}_w$ is a nested set on $\Gamma$ is easily verified by induction.

\begin{example}
Let $w$ be the element in the Weyl group $W$ of type $B_4$ from Example \ref{ex:sep-in-B4}.  Let $\Gamma$ be the Dynkin diagram, whose vertices are the simple roots $\alpha_1,\alpha_2, \alpha_3$ and $\alpha_4$.  Then $w$ corresponds to the nested set 
\[
\mathcal{N}_w=\{\{\alpha_1, \alpha_2, \alpha_3, \alpha_4\},\{\alpha_1, \alpha_2\},\{\alpha_2\},\{\alpha_4\}\}.
\]
The degrees of $W$ are $2,4,6,$ and $8$, and a Weyl group of type $A_{n-1}$ has degrees $2,3,\ldots,n$, thus part (2) of Theorem \ref{thm:bijection-to-faces} implies that 
\[
\Lambda_w^L(q)=q^{13}\Lambda^R_w(q^{-1})=\frac{\left([2]_q [4]_q [6]_q [8]_q \right) \left( [2]_q \right)}{\left([2]_q [3]_q\right) \left([2]_q \right)}=[4]_q [8]_q (1+q^3).
\]
\end{example}

\subsection{Proof of Theorem \ref{thm:bijection-to-faces}}
\label{sec:proof-of-bijection}

Throughout this section we assume for convenience that $W$ is irreducible, so the Dynkin diagram $\Gamma$ is connected.  We lose no generality by this assumption since:
\begin{itemize}
    \item nested sets on $\Gamma_1 \sqcup \Gamma_2$ are all of the form $\mc{N}_1 \sqcup \mc{N}_2$ with $\mc{N}_i$ a nested set on $\Gamma_i$,
    \item separable elements in $W_1 \times W_2$ are just pairs $(w_1,w_2)$ with $w_i$ a separable element in $W_i$, and 
    \item $\Lambda_{(w_1,w_2)}(q)=\Lambda_{w_1}(q)\Lambda_{w_2}(q)$.
\end{itemize}

\begin{lem} \label{lem:nested-without-maximal}
Let $\mc{N}$ be a nested set on a connected graph $\Gamma$ with $\Gamma \not \in \mc{N}$, then there is some $x \in \Gamma$ which is not contained in any element of $\mc{N}$.
\end{lem}
\begin{proof} 
Supposing otherwise, start with any element $y \in \Gamma$ and let $J_y$ be the maximal element of $\mc{N}$ containing $y$ (this exists by (N2)).  Let $y'$ be an element of $J_y$ which has an edge $\overline{y'z}$ with $z \not \in J_y$; such a $y'$ exists since $J_y \neq \Gamma$.  Now $z \in J_z$ for some $J_z \in \mc{N}$.  We may not have $J_z \subseteq J_y$, since $z \not \in J_y$, nor $J_y \subseteq J_z$, by the maximality of $J_y$, thus $J_z$ and $J_y$ are disjoint.  By (N1), we have that $\Gamma|_{J_y}$ and $\Gamma|_{J_z}$ are connected, but this implies that $\Gamma|_{J_y \cup J_z}$ is connected, since $\overline{y'z}$ is an edge, contradicting (N3), so the desired element $x$ exists.  
\end{proof}

\begin{lem} \label{lem:lands-in-separables}
For any nested set $\mc{N}$ on $\Gamma$, the element $w(\mc{N}) \in W$ is separable.
\end{lem}
\begin{proof}
Suppose by induction that this is true for all Weyl groups of lower rank than $W$, the base case of type $A_1$ being trivial, and for all nested sets $\mc{N}'$ on $\Gamma$ of smaller cardinality than $\mc{N}$, the case of $\mc{N}'=\emptyset$ corresponding to the identity element, which is separable.  


If $\Gamma \in \mc{N}$, let $\mc{N}'=\mc{N}\setminus \{\Gamma\}$, is clearly a nested set.  By induction, we know that $w(\mc{N}')$ is separable, and since $\Gamma$ appears first in any monotonic order we have $w(\mc{N})=w_0(\Gamma)w(\mc{N}')$.  Since $w(\mc{N}')$ is separable by induction, we conclude $w(\mc{N})$ is separable by Corollary \ref{cor:closure}.

Otherwise $\Gamma \not \in \mc{N}$, so by Lemma \ref{lem:nested-without-maximal} there is some $x \in \Gamma$ not contained in any element of $\mc{N}$. Now note that $w(\mc{N})$ in fact lies in the parabolic subgroup $W_{\Gamma \setminus \{x\}}$ and $\mc{N}$ is a nested set on $\Gamma \setminus \{x\}$, thus $w(\mc{N})$ is a separable element of $W_{\Gamma \setminus \{x\}}$ by induction.  It is immediate from Definition \ref{def:separable} that the inclusion of parabolic subgroup into a Weyl group sends separable elements to separable elements, so $w(\mc{N})$ is separable as an element of $W$.
\end{proof}

\begin{proof}[Proof of Theorem \ref{thm:bijection-to-faces}]
To prove (1), we check that $w(\mc{N}_u)=u$ for $u$ separable.  The other direction, that $\mc{N}_{w(\mc{N}')}=\mc{N}'$ is similar.  Again, assume the desired statement has already been proven for Weyl groups of smaller rank than $W$.  If $u$ has an empty pivot $\alpha$, then by construction $\mc{N}_u$ does not contain $\Gamma$ and 
\[
\mc{N}_u=\bigsqcup_{i=1}^k \mc{N}_{u_{\Gamma_i}}
\]
where $\{\Gamma_i\}$ are the connected components of $\Gamma \setminus \{\alpha \}$.  As there are no edges between connected components (by definition) it is clear that $w_0(J)$ and $w_0(J')$ commute for $J \subseteq \Gamma_i$ and $J' \subseteq \Gamma_{i'}$ with $i \neq i'$.  Thus
\[
    w(\mc{N}_u)=\prod_{i=1}^k w(\mc{N}_{u_{\Gamma_i}})=\prod_{i=1}^k u_{\Gamma_i}=u.
\]
If $u$ has a full pivot $\alpha$, then 
\[
w(\mc{N}_u)=w(\{\Gamma\} \sqcup \mc{N}_{w_0u})=w_0w(\mc{N}_{w_0u})=w_0^2u=u.
\]
Where here we have used the first case and the fact that $w_0$ is an involution.  

The factor of $2^r$ in the second statement of (1) occurs because we may optionally remove any of the $r$ maximal elements of the nested sets indexing faces of $A(\Gamma)$ (see Proposition \ref{prop:faces-are-nested}).

Let $F_{\mc{N}}(q)$ denote the quotient on the right hand side of (2).  To prove (2), we will see that both $\Lambda^L_{w(\mc{N}}(q)$ and $F_{\mc{N}}(q)$ satisfy the same recurrence.  Indeed, if $\Gamma \not \in \mc{N}$ then there is some $\alpha$ not contained in any element of $\mc{N}$ and, by the construction of the inverse bijection, $\alpha$ is an empty pivot of $w(\mc{N})$.  Let $\{\Gamma_i\}$ be the connected components of $\Gamma \setminus \{\alpha\}$, and $\{\mc{N}_i\}$ the nested sets on these components given by restricting $\mc{N}$.  Then 
\begin{align*}
F_{\mc{N}}(q)&=\prod_i F_{\mc{N}_i}(q),  \\
\Lambda^L_{w(\mc{N})}(q)&=\prod_i \Lambda^L_{w(\mc{N})_{\Gamma_i}}(q) = \prod_i \Lambda^L_{w(\mc{N}_i)}(q).
\end{align*}
Otherwise $\Gamma \in \mc{N}$ and we clearly have 
\[
F_{\mc{N}}(q)=\frac{W(q)}{F_{\mc{N}\setminus \{\Gamma\}}(q)}=\frac{W(q)}{F_{\mc{N}_{w_0w(\mc{N})}}(q)}.
\]
On the other hand, by Theorem \ref{thm:symmetric-and-unimodal} we have 
\[
\Lambda^L_{w(\mc{N})}(q)=\frac{W(q)}{V^L_{w(\mc{N})}(q)}=\frac{W(q)}{\Lambda^L_{w_0w(\mc{N})}(q)}.
\]
The claim $\Lambda^L_w(q)=q^{\ell(w)}\Lambda^R_w(q^{-1})$ is true for any element $w$: for $a \leq_L w$, we can write $ba=w$ with lengths adding, but this means $b \leq_R w$ and $\ell(b)=\ell(w)-\ell(a)$.
\end{proof}


\section{Proof of Theorem \ref{thm:only-if}}
\label{sec:proof-of-only-if}

In Propositions \ref{prop:splitting-descends-to-parabolic} and \ref{prop:splitting-iff-flipped-is} we give several methods of producing more splittings from a given one; these will be useful in the proof of Theorem \ref{thm:only-if}.

\begin{prop} \label{prop:splitting-descends-to-parabolic}
Let $(X,Y)$ be a splitting of a Weyl group $W$, then:
\begin{enumerate}
    \item $X$ and $Y$ have unique maximal elements $x_0$ and $y_0$ under left and right weak order respectively.  Furthermore, we have $x_0y_0=w_0$.
    \item Let $J \subseteq \Delta$, then $(X \cap W_J, Y \cap W_J)$ is a splitting of $W_J$.
\end{enumerate}
\end{prop}
\begin{proof}
By the definition of splitting, there exist unique elements $x_0 \in X$ and $y_0 \in Y$ such that $x_0y_0=w_0$.  Since all products $xy$ with $x \in X$ and $y \in Y$ are length-additive, we must in particular have $X \subseteq W/\{y_0\}$.  This generalized quotient, by Proposition \ref{prop:parabolic-quotient} is equal to $[e,w_0y_0^{-1}]_L=[e,x_0]_L$.  Similarly, by symmetry, we have $Y \subseteq [e,y_0]_R$.  This proves (1).

Let $X'=X \cap W_J$ and $Y'=Y \cap W_J$.  Since $W_J$ is a subgroup, all products from $X' \times Y'$ lie in $W_J$, and since the length function on parabolic subgroups agrees with that of the full group, these products are still length-additive.  It remains to check that the restricted map still surjects onto $W_J$.  Let $w \in W_J$ and let $(x,y) \in X \times Y$ be the unique pair with $xy=w$.  Let $x=s_{i_1} \cdots s_{i_k}$ and $y=s_{j_1} \cdots s_{j_{\ell}}$ be reduced expressions; by length-additivity $w=s_{i_1} \cdots s_{i_k} s_{j_1} \cdots s_{j_{\ell}}$ is a reduced expression.  Because $w \in W_J$, all of these simple reflections must in fact lie in $W_J$, and so $x,y \in W_J$.  Thus the restricted map is still surjective, and $(X',Y')$ is a splitting of $W_J$, proving (2).
\end{proof}

\begin{prop} \label{prop:splitting-iff-flipped-is}
Let $(X,Y)$ be a splitting of a Weyl group $W$, and let $x_0$ and $y_0$ denote the left- and right-maximal elements of $X,Y$ respectively, as given by Proposition \ref{prop:splitting-descends-to-parabolic}.  Define maps $\varphi:W \to W$ and $\psi:W \to W$ by 
\begin{align*}
    \varphi(u) &= ux_0^{-1} \\
    \psi(u) &= w_0y_0^{-1}uw_0=x_0uw_0.
\end{align*}
Then $(\varphi(X),\psi(Y))$ is also a splitting of $W$, having left- and right-maximal elements $x_0^{-1}$ and $x_0w_0=w_0y_0^{-1}w_0$ respectively.
\end{prop}
\begin{proof}
For $u \in [e,x_0]_L$ we can write $zu=x_0$ with lengths adding; by taking inverses we see that $\ph(u)=ux_0^{-1}=z^{-1}\in [e,x_0^{-1}]_L$, so $\ph$ restricts to a bijection $[e,x_0]_L \to [e,x_0^{-1}]_L$.  Similarly, $\psi$ restricts to a bijection $[e,y_0]_R \to [e,x_0w_0]_R$.  Since $\psi(Y) \subseteq [e,x_0w_0]_R$ and $\ph(X) \subseteq [e,x_0^{-1}]_L=W/[e,x_0w_0]_R$ we see that products of elements from $\ph(X)$ and $\psi(Y)$ are length-additive.  Given $x \in X$ and $y \in Y$ we have $\ph(x)\psi(y)=xx_0^{-1} \cdot x_0 y w_0=xyw_0$, so the bijectivity of the multiplication map $\ph(X) \times \psi(Y) \to W$ follows from that for $X \times Y \to W$.  Thus $(\ph(X),\psi(Y))$ is a splitting.
\end{proof}

The proof of Theorem \ref{thm:only-if} will proceed by studying potential minimal counterexamples.  We say $w \in W$ is a \emph{minimal non-separable} element if $w$ is not separable, but $w_J \in W_J$ is separable for all $J \subsetneq \Delta$.

We now restrict our attention to symmetric groups $W=S_n$. The following lemma describes the structure of minimal non-separable permutations.
\begin{lem}\label{lem:minimal-not-sep-structure}
Let $w\in S_n$ be minimal non-separable. If $w(1)<w(n)$, then $w(i)>w(n)$ for $2\leq i\leq n+1-w(n)$; $w(1)<w(i)<w(n)$ for $n+2-w(n)\leq i\leq n-w(1)$; and $w(i)<w(1)$ for $n-w(1)+1\leq i\leq n-1$. Likewise, if $w(1)>w(n)$, then $w(i)<w(n)$ for $2\leq i\leq w(n)$; $w(n)<w(i)<w(1)$ for $w(n)+1\leq i\leq w(1)-1$; and $w(i)>w(1)$ for $w(1)\leq i\leq n-1$. 
\end{lem}

The structure of minimal non-separable permutations described in Lemma~\ref{lem:minimal-not-sep-structure} can be viewed in Figure~\ref{fig:min-non-sep}.

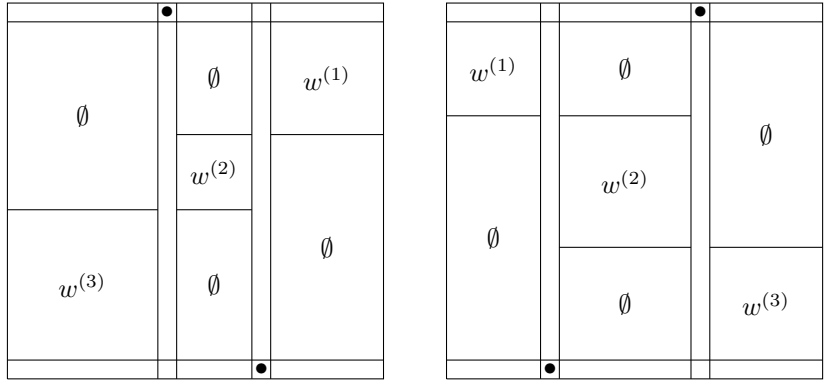
\begin{figure}[h!]
\centering
\begin{tikzpicture}[scale=0.25]
\draw(0,0)--(0,20)--(20,20)--(20,0)--(0,0);
\draw(0,1)--(20,1);
\draw(0,19)--(20,19);
\draw(8,0)--(8,20);
\draw(9,0)--(9,20);
\draw(13,0)--(13,20);
\draw(14,0)--(14,20);
\node at (8.5,19.5) {$\bullet$};
\node at (13.5,0.5) {$\bullet$};
\draw(0,9)--(8,9);
\draw(9,9)--(13,9);
\draw(9,13)--(13,13);
\draw(14,13)--(20,13);
\node at (4,14) {$\emptyset$};
\node at (4,5) {$w^{(3)}$};
\node at (11,5) {$\emptyset$};
\node at (11,11) {$w^{(2)}$};
\node at (11,16) {$\emptyset$};
\node at (17,7) {$\emptyset$};
\node at (17,16) {$w^{(1)}$};
\end{tikzpicture}
\qquad
\begin{tikzpicture}[scale=0.25]
\draw(0,0)--(0,20)--(20,20)--(20,0)--(0,0);
\draw(0,1)--(20,1);
\draw(0,19)--(20,19);
\draw(5,0)--(5,20);
\draw(6,0)--(6,20);
\draw(13,0)--(13,20);
\draw(14,0)--(14,20);
\node at (5.5,0.5) {$\bullet$};
\node at (13.5,19.5) {$\bullet$};
\draw(0,14)--(5,14);
\draw(6,14)--(13,14);
\draw(6,7)--(13,7);
\draw(14,7)--(20,7);
\node at (2.5,7.5) {$\emptyset$};
\node at (2.5,16.5) {$w^{(1)}$};
\node at (9.5,4) {$\emptyset$};
\node at (9.5,10.5) {$w^{(2)}$};
\node at (9.5,16.5) {$\emptyset$};
\node at (17,4) {$w^{(3)}$};
\node at (17,13) {$\emptyset$};
\end{tikzpicture}
\caption{Structure of minimal non-separable permutations (viewed as the permutation matrix with $(i,w(i))$ labeled for $i=1,\ldots,n$): left for $w(1)<w(n)$ and right for $w(1)>w(n)$.}
\label{fig:min-non-sep}
\end{figure}

\begin{proof}[Proof of Lemma~\ref{lem:minimal-not-sep-structure}]
Let $w\in S_n$ be a minimal non-separable permutation. By definition, $w(1),\ldots,w(n-1)$ is separable and $w(2),\ldots,w(n)$ is also separable. Thus any occurrence in $w$ of either forbidden pattern 3142 or 2413 must use $w(1)$ and $w(n)$. Depending on whether $w(1)<w(n)$ or $w(1)>w(n)$, exactly one of the patterns 3142 and 2413 can appear in $w$. Without loss of generality, let us assume that $w(1)<w(n)$ so that $w$ contains 2413 and $w$ does not contain 3142. The case $w(1)>w(n)$ can be obtained by reversing the permutation.

We first notice that $w(1),w(2),w(n{-}1),w(n)$ form a pattern 2413. To see this, find indices $i<j$ such that $w(1),w(i),w(j),w(n)$ form 2413, with smallest possible $k:=(i-1)+(n-j)$. If $i\neq2$, consider the value $w(2)$. If $w(2)>w(n)$, then we can replace $i$ by 2, decreasing $k$. If $w(j)<w(2)<w(n)$, then $w$ contains a pattern 2413 at indices $2,i,j,n$, contradicting $w$ being minimal non-separable. If $w(2)<w(j)$, then $w$ contains a pattern 3142 at indices $1,2,i,j$, contradicting $w$ being minimal non-separable. As a result, $i=2$, and analogously, $j=n-1$.

Next, we partition $\{3,4,\ldots,n-2\}$ into sets as follows:
\begin{itemize}
    \item $A_1'=\{i\ |\ w(i)<w(n-1)\}$, $A_1''=\{i\ |\ w(n-1)<w(i)<w(1)\}$,
    \item $A_2=\{i\ |\ w(1)<w(i)<w(n)\}$,
    \item $A_3''=\{i\ |\ w(n)<w(i)<w(2)\}$, $A_3'=\{i\ |\ w(i)>w(2)\}$.
\end{itemize}
These sets are depicted in Figure~\ref{fig:min-non-sep-pf}. Also, let $A_1=A_1'\cup A_1''\cup\{n-1\}$ and $A_3=A_3'\cup A_3''\cup\{2\}$. We show that if $a\in A_i$ and $b\in A_j$ with $i<j$, then $a>b$. For the sake of contradiction, assume that $a\in A_i$, $b\in A_j$, $i<j$ and $a<b$. Since either $i=1$ or $j=3$, by symmetry, let us assume that $i=1$. If $a\in A_1'$, then $w$ contains a pattern 3142 at indices $1,a,b,n-1$, contradicting $w$ being minimal non-separable. If $a\in A_1''$ and $b\in A_2\cup A_3''$, then $w$ contains a pattern 2413 at indices $1,2,a,b$ and if $a\in A_1''$ and $b\in A_3'$, then $w$ contains a pattern 3142 at indices $2,a,b,n$, contradicting $w$ being minimal non-separable. And if $a=n-1$, we simply cannot have $a<b$. By definition, $A_1=\{i\ |\ w(i)<w(1)\}$ but the above argument also shows that $A_1$ contains the largest indices in $\{2,3,\ldots,n-1\}$. Thus, $A_1=\{n-w(1)+1,\ldots,n-1\}$. Similarly, we deduce that $A_3=\{2,\ldots,n+1-w(n)\}$ and $A_2=\{n+2-w(n),\ldots,n-w(1)\}$. This is exactly what we want.
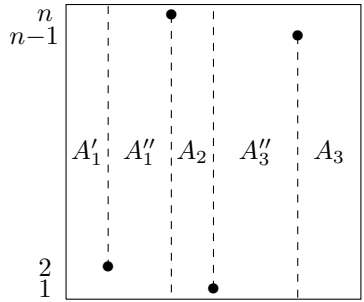
\begin{figure}[h!]
\centering
\begin{tikzpicture}[scale=0.28]
\draw(0,0)--(0,14)--(14,14)--(14,0)--(0,0);
\node at (2,1.5) {$\bullet$};
\node at (5,13.5) {$\bullet$};
\node at (7,0.5) {$\bullet$};
\node at (11,12.5) {$\bullet$};
\draw[dashed](2,1.5)--(2,14);
\draw[dashed](5,13.5)--(5,0);
\draw[dashed](7,0.5)--(7,14);
\draw[dashed](11,12.5)--(11,0);
\node at (1,7) {$A_1'$};
\node at (3.5,7) {$A_1''$};
\node at (6,7) {$A_2$};
\node at (9,7) {$A_3''$};
\node at (12.5,7) {$A_3$};
\node at (-1,0.5) {$1$};
\node at (-1,1.5) {2};
\node at (-1.5,12.5) {$n{-}1$};
\node at (-1,13.5) {$n$};
\end{tikzpicture}
\caption{A partition of permutation entries for minimal non-separable permutations.}
\label{fig:min-non-sep-pf}
\end{figure}
\end{proof}

The following lemma is the main technical lemma of this section.

\begin{lem} \label{lem:minimal-not-symmetric}
Let $w\in S_n$ be minimal non-separable and $k=|w(1)-w(n)|-1$. Let $f=[e,w]_L(q)$ be the rank generating function of the weak interval below $w$, a polynomial in $q$ of degree $\ell(w)$. Then $[q^d]f=[q^{\ell(w)-d}]f$ for $d\leq k$, but $[q^{k+1}]f\neq[q^{\ell(w)-k-1}]f$. In particular, $f=[e,w]_L(q)$ is not palindromic.
\end{lem}
\begin{proof}
Visualization from Figure~\ref{fig:min-non-sep} is very helpful for the proof. Unlike in Lemma~\ref{lem:minimal-not-sep-structure}, the two cases $w(1)<w(n)$ and $w(n)>1$ are significantly different here.
Recall that, by Proposition \ref{prop:weak-given-by-inversions}, $u\leq_L \pi$ if and only if for all $i<j$ such $\pi(i)<\pi(j)$, we have $u(i)<u(j)$.

First, assume $w(1)>w(n)$ (see the right side of Figure~\ref{fig:min-non-sep}). Let $x=w(n)-1$ and $y=n-w(1)$. Then $x\geq1$, $y\geq1$, $k\geq0$ and $x+y+k=n-2$. From Lemma~\ref{lem:minimal-not-sep-structure}, we know $w$ maps $A_1:=\{2,\ldots,x+1\}$ to $\{1,\ldots,x\}$, $A_2:=\{x+2,\ldots,x+k+1\}$ to $\{x+2,\ldots,x+k+1\}$ and $A_3:=\{n-y,\ldots,n-1\}$ to $\{n-y+1,\ldots,n\}$. We view $w$ as three permutations $w^{(1)}\in S_x$, $w^{(2)}\in S_k$ and $w^{(3)}\in S_y$, by keeping the relative orderings of the indices. Since $w$ is minimal non-separable, $w^{(1)}$, $w^{(2)}$ and $w^{(3)}$ are separable. Let us describe a general element $u\in[e,w]_L$. For any permutation $u\in S_n$, let $u^{(1)}\in S_x$, $u^{(2)}\in S_k$ and $u^{(3)}\in S_y$ be obtained by restricting $u$ to indices $A_1$, $A_2$ and $A_3$ respectively, while keeping the relative orderings of the indices. Then we see that $u\in [e,w]_L$ if and only if $u(1)\leq w(1)$, $u(n)\leq w(n)$, $u(a)<u(b)$ if $a\in A_i$ and $b\in A_j$ with $i<j$, and $u^{(i)}\leq_L w^{(i)}$ for $i=1,2,3$. For simplicity, let $f_i:=[e,w^{(i)}]_L(q)$ be the rank generating function of the weak interval below $w^{(i)}$, for $i=1,2,3$. Since $w^{(i)}$ is separable, $f_i$ is palindromic by Theorem~\ref{thm:symmetric-and-unimodal}.

We calculate that $\ell(w)=\ell(w^{(1)})+\ell(w^{(2)})+\ell(w^{(3)})+x+2k+y+1$. For $d\leq k+1$, and $u\in[e,w]_L$ with $\ell(u)=\ell(w)-d$, let $s=w(1)-u(1)\geq0$ and $t=u(n)-w(n)\geq0$. For $d\leq k$, we necessarily have $s+t\leq k$. When $s,t$ with $s+t\leq k$ are fixed, the entries $u(1)$ and $u(n)$ have already contributed $(x+k-s)+(y+k-t)+1$ to $\ell(u)$. So when $d\leq k$, we have
\[
[q^{\ell(w)-d}]f=\sum_{\substack{s,t\geq0\\s+t\leq k}}[q^{\ell(w^{(1)})+\ell(w^{(2)})+\ell(w^{(3)})-d+s+t}]f_1f_2f_3=\sum_{\substack{s,t\geq0\\s+t\leq k}}[q^{d-s-t}]f_1f_2f_3,
\]
where the last step follows because $f_1f_2f_3$ is palindromic. To compute $[q^d]f$ when $d\leq k+1$, we sum over $s\geq0$ and $t\geq0$ with $v(1)=s+1$ and $v(n)=n-t$. As $x,y\geq1$, $n\geq k+4$ and $v(n)>v(1)$. The entries $v(1)$ and $v(n)$ have contributed $s+t$ to $\ell(v)$, so we have
$$[q^{d}]f=\sum_{s,t\geq0}[q^{d-s-t}]f_1f_2f_3.$$
This means $[q^{\ell(w)-d}]f=[q^d]f$ for $d\leq k$. When $k=d+1$, the calculation for $[q^{\ell(w)-d}]f$ remains the same when $s+t\leq k$ (that is, when $u(1)>u(n)$). But now it is possible that $u(1)<u(n)$, $u\in[e,w]_L$ and $\ell(u)=\ell(w)-k-1$. Such $u$'s must satisfy $u^{(i)}=w^{(i)}$ for $i=1,2,3$, $u(n)-u(1)=1$, and $x+1\leq u(1)\leq x+k+1$. There are precisely $k+1$ possibilities. Thus,
$$[q^{\ell(w)-k-1}]f=\sum_{\substack{s,t\geq0\\s+t\leq k}}[q^{k+1-s-t}]f_1f_2f_3+k+1.$$
The calculation for $[q^{d}]f$ stays the same so
\begin{align*}
[q^{k+1}]f=&\sum_{s,t\geq0}[q^{k+1-s-t}]f_1f_2f_3=\sum_{\substack{s,t\geq0\\s+t\leq k}}[q^{k+1-s-t}]f_1f_2f_3+\sum_{\substack{s,t\geq0\\s+t=k+1}}1\\
=&\sum_{\substack{s,t\geq0\\s+t\leq k}}[q^{k+1-s-t}]f_1f_2f_3+k+2\\
=&[q^{\ell(w)-k-1}]f+1.
\end{align*}

Next, assume $w(1)<w(n)$ (see the left side of Figure~\ref{fig:min-non-sep}). Let $x=n-x(1)\geq1$ and $y=x(n)-1\geq1$. From Lemma~\ref{lem:minimal-not-sep-structure}, we know that $w$ maps $A_1:=\{2,\ldots,x+1\}$ to $\{n-x+1,\ldots,n\}$, $A_2:=\{x+2,\ldots,x+k+1\}$ to $\{n-x-k,\ldots,n-x-1\}$ and $A_3:=\{n-y,\ldots,n-1\}$ to $\{1,\ldots,y\}$. As before, for $u\in S_n$, let $u^{(1)}\in S_x$, $u^{(2)}\in S_k$, $u^{(3)}\in S_y$ be obtained by restricting $u$ to $A_1$, $A_2$ and $A_3$ respectively. In this case, $u\leq_L w$ if and only if $u(1)<u(a)$ for $a\in A_1\cup A_2\cup\{n\}$, $u(n)>u(b)$ for $b\in \{1\}\cup A_2\cup A_3$, and $u^{(i)}\leq w^{(i)}$ for $i=1,2,3$. We are going to compute $f$ explicitly. 

Notice that $\ell(w)=xy+(k+1)(x+y)+\ell(w^{(1)})+\ell(w^{(2)})+\ell(w^{(3)})$. We group together those $u\in[e,w]_L$ with the same values at $u(1)=w(1)-s$ and $u(n)=w(n)+t$, for $0\leq s\leq y$ and $0\leq t\leq x$. Now that $s$ and $t$ are fixed, we must have $\{i\ |\ u(i)>u(n)\}\subset A_1$ and $\{i\ |\ u(i)<u(1)\}\subset A_3$. We still need to assign $s$, $k$, and $t$ values from $\{u(1)+1,\ldots,u(n)-1\}$ to $u(A_1)$, $u(A_2)$ and $u(A_3)$ respectively. After we specify the sets of values $u(A_1)$, $u(A_2)$ and $u(A_3)$, we just need to make sure that each $u^{(i)}\leq_L w^{(i)}$. We enumerate inversions $(i,j)$ of $u$ by grouping them as follows: $i=1$ contributes $y-s$; $j=n$ contributes $x-t$; $i\in\{i\ |\ u(i)>u(n)\}$ contributes $(x-t)(y+k)$; $j\in\{j\ |\ u(j)<u(1)\}$ contributes an additional $(y-s)(k+t)$; the inversions within each $u^{(i)}$ contribute $\ell(u^{(i)})$ for $i=1,2,3$; and some additional inversions from assigning $\{u(1)+1,\ldots,u(n)-1\}$ to $u(A_1)$, $u(A_2)$ and $u(A_3)$.
Putting these together, we obtain
\begin{align*}
f=&\sum_{s=0}^y\sum_{t=0}^xq^{(x-t)(y+k+1)+(y-s)(k+t+1)}\qbinom{s+t+k}{s,t,k}_qf_1f_2f_3\\
=&\sum_{s=0}^y\sum_{t=0}^xq^{(k+1)(x+y-t-s)+xy-ts}\qbinom{s+t+k}{s,t,k}_qf_1f_2f_3.
\end{align*}
Let $f_{st}$ be the summand above such that $f=\sum_{s=0}^y\sum_{t=0}^x f_{st}.$ Let $d\leq k+1$. As $x,y\geq1$, as long as $s<y$ or $t<x$, we have $(k+1)(x+y-t-s)+xy-ts>d$ so $[q^d]f_{st}=0$. This means
$$[q^d]f=[q^d]\qbinom{x+y+k}{x,y,k}_qf_1f_2f_3,\quad d\leq k+1.$$

On the other hand, let us note that $\qbinom{s+t+k}{s,t,k}_q$ is palindromic of degree $st+tk+sk$. By symmetry of the $q$-multinomial coefficient and $f_1,f_2,f_3$, we have
\begin{align*}
[q^{\ell(w)-d}]f=&\sum_{s=0}^y\sum_{t=0}^x[q^{\ell(w)-d}]f_{st}\\
=&\sum_{s=0}^y\sum_{t=0}^x[q^{st+tk+sk+\ell(w^{(1)})+\ell(w^{(2)})+\ell(w^{(3)})+s+t-d}]\qbinom{s+t+k}{s,t,k}_qf_1f_2f_3\\
=&\sum_{s=0}^y\sum_{t=0}^x[q^{d-s-t}]\qbinom{s+t+k}{s,t,k}_qf_1f_2f_3\\
=&[q^d]\sum_{s=0}^y\sum_{t=0}^xq^{s+t}\qbinom{s+t+k}{s,t,k}_qf_1f_2f_3.
\end{align*}

For the sake of computation, let us show the following equality combinatorially:
\begin{equation} \label{eq:multinomial-congruence}
\qbinom{s+t+k}{s,t,k}_q\equiv \qbinom{s+k}{k}_q\qbinom{t+k}{k}_q\bmod{q^{k+1}}.
\end{equation}
Recall that $\qbinom{s+t+k}{s,t,k}_q$ is the generating function of arrangements of $s$ 1's, $k$ 2's and $t$ 3's graded by the number of inversions. The left hand side of the above equation counts the number of such arrangement with at most $k$ inversions, graded by the number of inversions. We see that if a 3 appears before any 1's, then the sequence has at least $k+1$ inversions. Since we are working modulo $q^{k+1}$, these terms can be ignored. The right hand side counts $(\sigma,\tau)$, while $\sigma$ is an arrangement of $s$ 1's and $k$ 2's and $\tau$ is an arrangement of $k$ 2's and $t$ 3's such that their numbers of inversions sum up to at most $k$. If the right most 1 in $\sigma$ has $a$ 2's on its right and the left most 3 in $\tau$ has $b$ 2's on its left, we necessarily have $a+b\leq k$, the total number of 2's. Therefore, we can merge together $\sigma$ and $\tau$ by identifying their 2's unambiguously. Conversely, for any arrangement of $s$ 1's, $k$ 2's and $t$ 3's with at most $k$ inversions, we can form $\sigma$ by ignoring the 3's and $\tau$ by ignoring the 1's. Since there are no inversions between 3's and 1's, the length is also preserved. Thus, a bijection is provided to justify the above equation.

As a result, for $s,t\geq0$, we have
\[
q^{s+t}\qbinom{s+t+k}{s,t,k}_q\equiv q^{s+t}\qbinom{s+k}{k}_q\qbinom{t+k}{k}_q\bmod{q^{k+2}},
\]
which follows from (\ref{eq:multinomial-congruence}) when $s+t\geq1$, and is trivially true when $s=t=0$. Let us continue the computation as follows:
\begin{align*}
&\sum_{s=0}^y\sum_{t=0}^xq^{s+t}\qbinom{s+t+k}{s,t,k}_qf_1f_2f_3\\
\equiv&\sum_{s=0}^y\sum_{t=0}^xq^{s+t}\qbinom{s+k}{k}_q\qbinom{t+k}{k}_qf_1f_2f_3\\
\equiv&\left(\sum_{s=0}^yq^s\qbinom{s+k}{k}_q\right)\left(\sum_{t=0}^xq^t\qbinom{t+k}{k}_q\right)f_1f_2f_3\\
\equiv&\qbinom{y+k+1}{k+1}_q\qbinom{x+k+1}{k+1}_qf_1f_2f_3\\
\equiv&\qbinom{x+y+k+1}{k+1}_qf_1f_2f_3\\
\equiv&\qbinom{x+y+k}{x,y,k}_qf_1f_2f_3+q^{k+1}\qbinom{x+y+k}{k+1}_q\qbinom{x+y}{x}_qf_1f_2f_3 \quad\bmod{q^{k+2}}.
\end{align*}

By our previous computed values of $[q^d]f$ and $[q^{\ell(w)-d}]f$, we see that the quantity $[q^{\ell(w)-d}]f-[q^d]f$ is 0 when $d\leq k$, and is 1 when $d=k+1$. 
\end{proof}

We are now ready to prove Theorem~\ref{thm:only-if}.

\begin{proof}[Proof of Theorem \ref{thm:only-if}]
Let $W=S_n$ and suppose by induction that we have proved the theorem for smaller symmetric groups, and thus, since splittings respect taking products, for all parabolic subgroups of $W$, the base case being easy to check directly.  This means that any splitting $(X',Y')$ of $W_J$ for $J \subsetneq \Delta$ is of the form $X'=[e,w_0(J)u'^{-1}]_L$ and $Y=[e,u']_R$ with $u'$ separable.  

Suppose we have a splitting $(X,Y)$ of $W$ which is not of this form.  Let $x_0 \in X$ and $y_0 \in Y$ be as in Proposition \ref{prop:splitting-descends-to-parabolic}, so that $X \subseteq [e,x_0]_L$ and $Y \subseteq [e, y_0]_R$ with $x_0y_0=w_0$.  If $x_0$ (and therefore $y_0$) is separable, then we have a splitting $([e,x_0]_L, [e, y_0]_R)$ by Theorem \ref{thm:separable-splitting}, so no proper subsets of these intervals could give a splitting.  Thus $x_0$ and $y_0$ are not separable; we wish to conclude that they are minimal non-separable elements.  To see this, for each $J \subsetneq \Delta$, consider the splitting $(X \cap W_J, Y \cap W_J)$ of $W_J$.  By the inductive hypothesis, we know that $X \cap W_J=[e,x_0']_L$ and $Y \cap W_J=[e,y_0']_R$ for some separable elements $x_0', y_0'$ with $x_0'y_0'=w_0(J)$ and $\ell(x_0')+\ell(y_0')=\ell(w_0(J))$.  Now, since $x_0$ and $y_0$ are the unique maximal elements in $X,Y$ under left and right weak order respectively by Proposition \ref{prop:splitting-descends-to-parabolic}, we have $x_0' \leq_L x_0 \land_L w_0(J)$ and $y_0' \leq_R y_0 \land_R w_0(J)$.  Since $x_0 \land_L w_0(J) \in W/\{y_0\}=W/[e,y_0]_R$ and $y_0 \land_R w_0(J) \in [e,y_0]_R$, we know that the product of these two elements is length-additive, and also clearly lies in $W_J$.  Since the length of this product is at least the length of $x_0'y_0'=w_0(J)$, and since $w_0(J)$ is the longest element of $W_J$, we conclude that $x_0'=x_0 \land_L w_0(J)=x_{0,J}$ and $y_0'=y_0 \land_R w_0(J)$; thus $x_{0,J}$ is separable for all $J \subsetneq \Delta$ and so $x_0$ is a minimal non-separable element. 

By Proposition \ref{prop:product-of-degrees}, we know $W(q)$ is the product of the $q$-integers $[d_i]_q=1+q+\cdots +q^{d_i-1}$ over the degrees $d_i$ of $W$, and is thus a product of cyclotomic polynomials (which are irreducible over $\mathbb{Q}$).  The splitting $(X,Y)$ gives a factorization $W(q)=X(q)Y(q)$ where $X(q)=\sum_{x \in X}q^{\ell(x)}$ and similarly for $Y(q)$, and so $X(q)$ is also a product of cyclotomic polynomials.  Since cyclotomic polynomials are known to have symmetric sequences of coefficients, $X(q)$ also has symmetric coefficients.  We know from Proposition \ref{prop:splitting-descends-to-parabolic} that $X \subseteq [e,x_0]_L$, from the above argument that $x_0$ is a minimal non-separable element, and from Lemma \ref{lem:minimal-not-symmetric} that $[e,x_0]_L(q)$ is not symmetric.  In what remains of the proof, we show that $X$ agrees with $[e,x_0]_L$ for elements of low or high length, thus implying that $X(q)$ is not symmetric and reaching a contradiction.

Above we showed that $X \cap W_J = [e,x_0']_L=[e,x_{0,J}]_L$ for any $J \subsetneq \Delta$, thus
\[
\bigcup_{J \subsetneq \Delta} [e, x_{0,J}]_L = \bigcup_{J \subsetneq \Delta} \left([e, x_{0}]_L \cap W_J\right) \subseteq X.
\]
The Weyl group $W$ has rank $n-1$, and so any element of length at most $n-2$ lies in some parabolic subgroup.  Therefore $[e,x_0]_L$ and $X$ agree for elements of length $0,1,...,n-2$.  We can apply this same argument to the splitting $(\ph(X),\psi(Y))$, so $\ph(X)$ and $[e,x_0^{-1}]_L$ also agree for elements of small length.  But the map $\ph: [e,x_0]_L \to [e,x_0^{-1}]_L$ satisfies $\ell(\ph(u))=\ell(x_0)-\ell(u)$.  Thus we conclude that $X$ and $[e,x_0]_L$ also agree for elements of length $\ell(x_0),...,\ell(x_0)-(n-2)$.  By Lemma \ref{lem:minimal-not-symmetric} this is enough to imply that $X(q)$ is not symmetric, since the quantity $k$ in Lemma \ref{lem:minimal-not-symmetric} is at most $n-3$ for a minimal non-separable permutation.
\end{proof}

\begin{remark}
    In this section we have not relied on the assumption $W=S_n$ except in Lemmas \ref{lem:minimal-not-sep-structure} and \ref{lem:minimal-not-symmetric}, which describe the structure of minimal non-separable permutations and the order ideals they generate in weak order.  Thus a uniform proof of Conjecture \ref{conj:other-types-only-if} may be obtainable from a type-independent understanding of minimal non-separable elements.
\end{remark}

\section{Proof of Theorem \ref{thm:surjectivity}}
\label{sec:proof-of-surjectivity}

The proof of Theorem~\ref{thm:surjectivity} relies on the following technical lemma.
\begin{lem}\label{lem:surjmain}
Let $w,\pi,u\in S_n$ such that $u\leq_L w$ and $u\leq_R\pi$. If 
\begin{equation} \label{eq:not-reduced}
    \ell(wu^{-1})+\ell(u)+\ell(u^{-1}\pi)>\ell(wu^{-1}\pi)
\end{equation}
then there exists $u'>_B u$ in the strong Bruhat order such that $u'\leq_L w$ and $u'\leq_R\pi$.
\end{lem}

We now observe that Theorem \ref{thm:surjectivity} follows from Lemma \ref{lem:surjmain}.

\begin{proof}[Proof of Theorem~\ref{thm:surjectivity}]
For any $x,y\in S_n$ and any maximal element $z$ in the Bruhat order such that $z\leq_L x$ and $z\leq_R y$, we have
\[
\ell(xz^{-1})+\ell(z)+\ell(z^{-1}y)=\ell(xz^{-1}y),
\]
since otherwise a strictly larger $z$ can be found via Lemma~\ref{lem:surjmain}.  In particular, there is at least one element $z \in S_n$ satisfying these conditions.

Now choose any $w\in S_n$ and we will show that $w\in (W/U)\cdot U$, where $U=[e,u]_R$. Let $z\in S_n$ be such that $z\leq_L w$, $z\leq_R u$ and $\ell(wz^{-1}u)=\ell(wz^{-1})+\ell(z)+\ell(z^{-1}u)$, as in the last paragraph. This implies that $\ell(wz^{-1})+\ell(u)=\ell(wz^{-1}u)$, so by definition $wz^{-1}\in W/U$. Then $w=(wz^{-1})\cdot z\in (W/U) \cdot U$ as desired.
\end{proof}

We finish with the rather technical proof of Lemma \ref{lem:surjmain}.

\begin{proof}[Proof of Lemma \ref{lem:surjmain}]
For $x_1,\ldots,x_k \in S_n$, we say that $x_1\cdots x_k$ is \textit{reduced} if \[
\ell(x_1\cdots x_k)=\ell(x_1)+\cdots+\ell(x_k).
\]
Thus (\ref{eq:not-reduced}) says exactly that $(wu^{-1})(u)(u^{-1}\pi)$ is not reduced. 

Our proof idea relies on wiring diagrams for permutations. The idea is to find a ``minimal" double crossing between two wires in a wiring diagram of $(wu^{-1})u(u^{-1}\pi)$ and then adjust the diagram accordingly, at least one double crossing must exist since the product is not reduced. 

Let $w'=wu^{-1}$ and $\pi'=u^{-1}\pi$. Then $w'u=w$ and $u\pi'=\pi$ are reduced and $w'u\pi'$ is not reduced. We will first reduce to the case where $\pi'$ has a single right descent by induction on $\ell(w')+\ell(\pi')$. As $\pi'$ is not the identity, $\pi'$ must have some right descent. Let $s_i$ be a right descent of $\pi'$ and write $\pi'=\pi^{(i)}s_i$, where $\ell(\pi')=1+\ell(\pi^{(i)})$. If $w'u\pi^{(i)}$ is not reduced, then by the induction hypothesis, we can find $u'>_B u$ where $u'\leq_L w$ and $u'\leq_R u\pi^{(i)}\leq_R\pi$ so we are done. Thus, we can assume that $w\pi^{(i)}$ is reduced for all right descents $s_i$ of $\pi$. At the same time, $\ell(w\pi')=\ell(w\pi^{(i)})-1=\ell(w)+\ell(\pi')-2$, so there exists a unique pair $1\leq a<b\leq n$ such that $\pi'(a)>\pi'(b)$ and $w\pi'(a)<w\pi'(b)$. Each right descent $s_i$ gives such a pair $a=i$, $b=i+1$. Therefore, $\pi'$ has a single right descent. By the same reasoning, we may also assume that $w'$ has a single left descent.

Suppose that the single right descent of $\pi'$ is $s_a$ and the single left descent of $w'$ is $s_b$. With the argument above, we can further assume that $(w')(u)(\pi's_a)$ is reduced and $(s_bw')(u)(\pi')$ is reduced. Together, we must have $w'u\pi'(a)=b$ and $w'u\pi'(a+1)=b+1$. Pictorially, choose any wiring diagrams for $w'$, $u$ and $\pi'$ and draw them side by side, label each wire by its right end point and view the wires as directed from right to left (see Figure~\ref{fig:surjmain}). The wires $a$ and $a+1$ are the only pair intersecting twice, once inside $\pi'$ and once inside $w'$. These two wires, drawn in black, enclose a region $R$. By the reducedness of $(w')(u)(\pi's_a)$ and $(s_bw')(u)(\pi')$, a wire entering $R$ from the top must leave from the bottom, and vice versa. Moreover, for two wires entering $R$ from the top, they can only intersect inside $u$, because $\pi'$ has a single right descent and $w'$ has a single left descent.

With the above intuition in mind, we now discuss the key construction. We first produce a set $I=\{a=i_0>i_1>\cdots>i_k\}$ by adding one element at a time. We start with $\{a=i_0\}$ and suppose that we have already obtained $i_0,\ldots,i_t$, for some $t\geq0$. Let $i_{t+1}$ be the largest number $i'<i_t$ such that $\pi'(i')>\pi'(a+1)$ and $u\pi'(i')<u\pi'(i_t)$. Notice that these two conditions automatically imply $\pi'(a+1)<\pi'(i_{t+1})<\pi'(i_t)$, $u\pi'(a+1)<u\pi'(i_{t+1})<u\pi'(i_t), b+1<w'u\pi'(i_{t+1})$, and $w'u\pi'(i_{t+1})<w'u\pi'(i_t)$ if $t\geq1$, since $\pi'$ has a single right descent, $w'$ has a single left descent and both $(w')(u)(\pi's_a)$ and $(s_bw')(u)(\pi')$ are reduced. Stop when no new elements can be found. It is possible that $k=0$. Pictorially, $I$ consists of wires $a$ and $i_1,\ldots,i_k$ in a decreasing order such that they all enter region $R$ from top right in $\pi'$ and leave region $R$ from bottom left in $w'$, such that the wire $i_{t+1}$ is chosen to be the closest one above that does not intersect the wire $i_t$. The wires $I\setminus\{a\}$ are colored red in the left side of Figure~\ref{fig:surjmain}. We then construct another set $J=\{a+1=j_0<j_1<\cdots<j_m\}$ in a similar way. Assume we have already constructed $\{a+1=j_0<j_1<\cdots<j_t\}$. Let $j_{t+1}$ be the smallest $j'>j_t$ such that $\pi'(j')<\pi'(a)$ and $u\pi'(j')>u\pi'(j_t)$. These conditions automatically imply that $\pi'(j_t)<\pi'(j_{t+1})<\pi'(a)$, $u\pi'(j_t)<u\pi'(j_{t+1})<u\pi'(a)$, $w'u\pi'(j_{t+1})<b$ and $w'u\pi'(j_{t+1})>w'u\pi'(j_t)$ if $t\geq1$. Similarly, in terms of wiring diagrams, $J$ consists of wire $a+1$ together with $j_1<\cdots<j_m$ such that all of them enter $R$ from bottom right in $\pi'$ and leave $R$ from top left in $w'$, such that the wire $j_{t+1}$ is chosen to be the closest one below $a+1$ that does not intersect the wire $j_t$. They are colored green in the left side of Figure~\ref{fig:surjmain}.

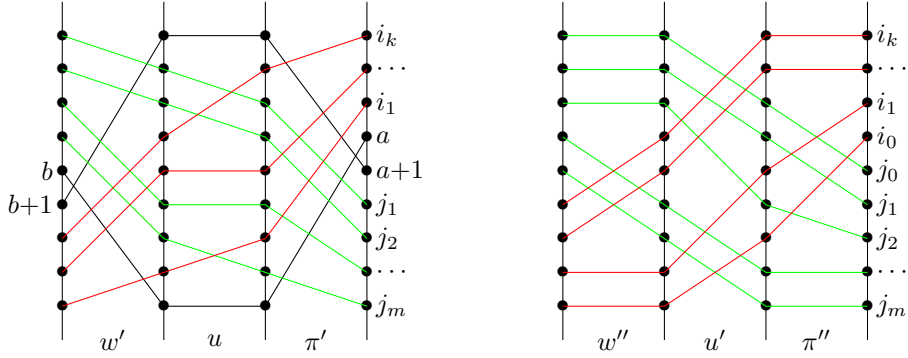
\begin{figure}[h!]
\centering
\begin{tikzpicture}[scale=0.45]
\pgfmathsetmacro{\K}{3}
\node at (0,-1) {$\bullet$};
\node at (0,-2) {$\bullet$};
\node at (0,-3) {$\bullet$};
\node at (0,-4) {$\bullet$};
\node at (0,-5) {$\bullet$};
\node[left] at (0,-5) {$b$};
\node at (0,-6) {$\bullet$};
\node[left] at (0,-6) {$b{+}1$};
\node at (0,-7) {$\bullet$};
\node at (0,-8) {$\bullet$};
\node at (0,-9) {$\bullet$};
\node at (\K,-1) {$\bullet$};
\node at (\K,-2) {$\bullet$};
\node at (\K,-3) {$\bullet$};
\node at (\K,-4) {$\bullet$};
\node at (\K,-5) {$\bullet$};
\node at (\K,-6) {$\bullet$};
\node at (\K,-7) {$\bullet$};
\node at (\K,-8) {$\bullet$};
\node at (\K,-9) {$\bullet$};
\node at (2*\K,-1) {$\bullet$};
\node at (2*\K,-2) {$\bullet$};
\node at (2*\K,-3) {$\bullet$};
\node at (2*\K,-4) {$\bullet$};
\node at (2*\K,-5) {$\bullet$};
\node at (2*\K,-6) {$\bullet$};
\node at (2*\K,-7) {$\bullet$};
\node at (2*\K,-8) {$\bullet$};
\node at (2*\K,-9) {$\bullet$};
\node at (3*\K,-1) {$\bullet$};
\node[right] at (3*\K,-1) {$i_k$};
\node at (3*\K,-2) {$\bullet$};
\node[right] at (3*\K,-2) {$\cdots$};
\node at (3*\K,-3) {$\bullet$};
\node[right] at (3*\K,-3) {$i_1$};
\node at (3*\K,-4) {$\bullet$};
\node[right] at (3*\K,-4) {$a$};
\node at (3*\K,-5) {$\bullet$};
\node[right] at (3*\K,-5) {$a{+}1$};
\node at (3*\K,-6) {$\bullet$};
\node[right] at (3*\K,-6) {$j_1$};
\node at (3*\K,-7) {$\bullet$};
\node[right] at (3*\K,-7) {$j_2$};
\node at (3*\K,-8) {$\bullet$};
\node[right] at (3*\K,-8) {$\cdots$};
\node at (3*\K,-9) {$\bullet$};
\node[right] at (3*\K,-9) {$j_m$};
\draw(0,0)--(0,-10);
\draw(\K,0)--(\K,-10);
\draw(2*\K,0)--(2*\K,-10);
\draw(3*\K,0)--(3*\K,-10);
\node at (\K/2,-10) {$w'$};
\node at (3*\K/2,-10) {$u$};
\node at (5*\K/2,-10) {$\pi'$};
\draw(0,-6)--(\K,-1)--(2*\K,-1)--(3*\K,-5);
\draw(0,-5)--(\K,-9)--(2*\K,-9)--(3*\K,-4);
\draw[green](0,-1)--(\K,-2)--(2*\K,-3)--(3*\K,-6);
\draw[green](0,-2)--(\K,-3)--(2*\K,-4)--(3*\K,-7);
\draw[green](0,-3)--(\K,-6)--(2*\K,-6)--(3*\K,-8);
\draw[green](0,-4)--(\K,-7)--(2*\K,-8)--(3*\K,-9);
\draw[red](0,-7)--(\K,-4)--(2*\K,-2)--(3*\K,-1);
\draw[red](0,-8)--(\K,-5)--(2*\K,-5)--(3*\K,-2);
\draw[red](0,-9)--(\K,-8)--(2*\K,-7)--(3*\K,-3);
\end{tikzpicture}
\qquad\qquad
\begin{tikzpicture}[scale=0.45]
\pgfmathsetmacro{\K}{3}
\node at (0,-1) {$\bullet$};
\node at (0,-2) {$\bullet$};
\node at (0,-3) {$\bullet$};
\node at (0,-4) {$\bullet$};
\node at (0,-5) {$\bullet$};
\node at (0,-6) {$\bullet$};
\node at (0,-7) {$\bullet$};
\node at (0,-8) {$\bullet$};
\node at (0,-9) {$\bullet$};
\node at (\K,-1) {$\bullet$};
\node at (\K,-2) {$\bullet$};
\node at (\K,-3) {$\bullet$};
\node at (\K,-4) {$\bullet$};
\node at (\K,-5) {$\bullet$};
\node at (\K,-6) {$\bullet$};
\node at (\K,-7) {$\bullet$};
\node at (\K,-8) {$\bullet$};
\node at (\K,-9) {$\bullet$};
\node at (2*\K,-1) {$\bullet$};
\node at (2*\K,-2) {$\bullet$};
\node at (2*\K,-3) {$\bullet$};
\node at (2*\K,-4) {$\bullet$};
\node at (2*\K,-5) {$\bullet$};
\node at (2*\K,-6) {$\bullet$};
\node at (2*\K,-7) {$\bullet$};
\node at (2*\K,-8) {$\bullet$};
\node at (2*\K,-9) {$\bullet$};
\node at (3*\K,-1) {$\bullet$};
\node[right] at (3*\K,-1) {$i_k$};
\node at (3*\K,-2) {$\bullet$};
\node[right] at (3*\K,-2) {$\cdots$};
\node at (3*\K,-3) {$\bullet$};
\node[right] at (3*\K,-3) {$i_1$};
\node at (3*\K,-4) {$\bullet$};
\node[right] at (3*\K,-4) {$i_0$};
\node at (3*\K,-5) {$\bullet$};
\node[right] at (3*\K,-5) {$j_0$};
\node at (3*\K,-6) {$\bullet$};
\node[right] at (3*\K,-6) {$j_1$};
\node at (3*\K,-7) {$\bullet$};
\node[right] at (3*\K,-7) {$j_2$};
\node at (3*\K,-8) {$\bullet$};
\node[right] at (3*\K,-8) {$\cdots$};
\node at (3*\K,-9) {$\bullet$};
\node[right] at (3*\K,-9) {$j_m$};
\draw(0,0)--(0,-10);
\draw(\K,0)--(\K,-10);
\draw(2*\K,0)--(2*\K,-10);
\draw(3*\K,0)--(3*\K,-10);
\node at (\K/2,-10) {$w''$};
\node at (3*\K/2,-10) {$u'$};
\node at (5*\K/2,-10) {$\pi''$};
\draw[green](0,-1)--(\K,-1)--(2*\K,-3)--(3*\K,-5);
\draw[green](0,-2)--(\K,-2)--(2*\K,-4)--(3*\K,-6);
\draw[green](0,-3)--(\K,-3)--(2*\K,-6)--(3*\K,-7);
\draw[green](0,-4)--(\K,-6)--(2*\K,-8)--(3*\K,-8);
\draw[green](0,-5)--(\K,-7)--(2*\K,-9)--(3*\K,-9);
\draw[red](0,-6)--(\K,-4)--(2*\K,-1)--(3*\K,-1);
\draw[red](0,-7)--(\K,-5)--(2*\K,-2)--(3*\K,-2);
\draw[red](0,-8)--(\K,-8)--(2*\K,-5)--(3*\K,-3);
\draw[red](0,-9)--(\K,-9)--(2*\K,-7)--(3*\K,-4);
\end{tikzpicture}
\caption{The initial wiring diagram (left) and the new wiring diagram used to construct $u'$ (right).}
\label{fig:surjmain}
\end{figure}

We are now ready to construct $u'$ from the above data of $I$ and $J$. We will simultaneously construct $w''$ and $\pi''$ such that $w''u'=w'u=w$ and $u'\pi''=u\pi'=\pi$. Let $\pi''(i_t)=\pi'(i_{t+1})$ for $t=0,\ldots,k-1$, $\pi''(i_k)=\pi'(a+1)$, $\pi''(j_t)=\pi'(j_{t+1})$ for $j=0,\ldots,m-1$ and $\pi''(j_m)=\pi'(a)$. Let $\pi''(c)=\pi'(c)$ if $c\notin I\cup J$. Then we let $u'=u\pi'(\pi'')^{-1}$ and $w''=w'u(u')^{-1}$. An example of this construction is given in Figure~\ref{fig:surjmain}. We remark that this construction does not preserve $w'u\pi'$, i.e., $w'u\pi'\neq w''u'\pi''$. 

Two things remain to be checked: that $u'>_B u$ in the Bruhat order and that $w''u'\pi''$ is reduced.

We first observe some further properties of $u$ and $u'$. Let $i_t'=\pi'(i_t)$ for $t=0,\ldots,k$ and $j_t'=\pi'(j_t)$ for $t=0,\ldots,m$. Let $I'=\{i_0'>\cdots>i_k'\}$, $J'=\{j_0'<\cdots<j_m'\}$ and $K'=I'\cup J'$. We say that an element $s$ has \textit{rank} $t$ in a set $K$ if $s$ is the $t^{th}$ smallest element in $K$. From construction, $i_0'$ is the largest element in $K'$, $j_0'$ is the smallest element in $K'$, $u(i_0')$ is the largest element in $u(K')$ and $u(j_0')$ is the smallest element in $u(K')$. Also, $u(i_0')>\cdots>u(i_k')$ and $u(j_0')<\cdots<u(j_m')$. Moreover, we then show that the rank of $u(j_t')$ in $u(K')$ is at most the rank of $j_t'$ in $K'$, for $t\geq1$. Let the rank of $j_t'$ in $K'$ be $t+x+1$; this means 
\[
i_k'<\cdots<i_{k-x+1}'<j_t'<i_{k-x}'<\cdots<i_1'.
\]
In the reduced decomposition $u\pi'=\pi$ we have $j_t>i_1>\cdots>i_{k-x}$ and $\pi'(j_t)<\pi'(i_{k-x})<\cdots<\pi'(i_1)$ so we need to have $u(j_t')<u(i_{k-x}')<\cdots<u(i_1')$. This means the rank of $u(j_t')$ in $u(K')$ is at most $t+x+1$. Dually, the rank of $u(i_t')$ in $u(K')$ is at least the rank of $i_t'$ in $K'$. In Figure~\ref{fig:surjmain}, we see that the green wires in region $u$ go from lower right to upper left and the red wires go from upper right to lower left.

To see that $u'>_Bu$, we recall some properties of $u$ and $u'$ and show that any pair of such permutations satisfy $u'>_Bu$:
\begin{itemize}
    \item there exists $I'=\{i_0'>\cdots>i_k'\}$ and $J'=\{j_0'<\cdots<j_m'\}$ such that $u'(c)=u(c)$ if $c\notin K':=I'\cup J'$;
    \item $u(i_0')>\cdots>u(i_k')$ and $u(j_0')<\cdots<u(j_m')$;
    \item $i_0'$ and $j_0'$ are the largest and smallest elements in $K'$ respectively;
    \item the rank of $u(i_t')$ in $u(K')$ is at least the rank of $i_t'$ in $K'$ for $t=0,\ldots,k$ and the rank of $u(j_t')$ in $u(K')$ is at most the rank of $j_t'$ in $K'$ for $t=0,\ldots,m$;
    \item $u'(i_t')=u(i_{t-1}')$ for $1\leq t\leq k$, $u'(i_0')=u(j_m')$, $u'(j_t')=u(j_{t-1}')$ for $1\leq t\leq m$ and $u'(j_0')=u(i_k')$.
\end{itemize}
We show how to obtain $u'$ from $u$ by applying some transpositions that increase the length after each step, using induction on $k+m$. When $k=m=0$, we have $i_0'>j_0'$, $u(i_0')>u(j_0)'$ and $u'(i_0')<u'(j_0')$, so $u'=u\cdot(i_0',j_0')$ which takes $u'$ above $u$ in the strong Bruhat order. For the induction step, we have that $i_0'$ is the largest element in $K'$ and $u(i_0')$ is the largest element in $u(K')$. The second largest element in $I'\cup J'$ is either $i_1'$ or $j_m'$. If $j_m'>i_1'$ (which is the case shown in Figure~\ref{fig:surjmain}), consider $u\cdot(j_m',i_0')$. As $j_m'<i_0'$ and $u(j_m')<u(i_0')$, we know $u<_B u\cdot(j_m',i_0')$ in the Bruhat order. Moreover, $u\cdot(j_m',i_0')(i_0')=u(j_m')=u'(i_0')$. We can then apply the induction hypothesis to the pair $u\cdot(j_m',i_0')$ and $u'$, with $j_m'$ deleted and $i_0'$ replaced by $j_m'$ (so $k$ stays the same and $m$ decreased by 1). The only nontrivial condition to check is the rank condition when the rank of $i_t'$ in $K'$ stays the same and the rank of $u(i_t')$ in $u(K')$ decreases. But this means that the rank of $i_t'$ in $K'$ is at least 1 less than the rank of $u(i_t')$ in $u(K')$. Thus, $u<_Bu\cdot(j_m',i_0')<_Bu'$ so we are done. If $j_m'<i_1'$, then similarly apply induction hypothesis to $u(i_1',i_0')$ and $u'$, with $i_0'$ deleted from $I'$. In both cases, the induction steps go through so we conclude that $u<_B u'$.

Finally, we check that $w''u'\pi''$ is reduced. This is the same as showing that for any $x,y$, the sequence 
\[
S=(x-y, \pi''(x)-\pi''(y), u'\pi''(x)-u'\pi''(y), w''u'\pi''(x)-w''u'\pi''(y))
\]
changes signs at most once. We say that $S$ changes signs in $\pi''$ if $x-y$ has a different sign than $\pi''(x)-\pi''(y)$ and so on. Recall that $I=\{i_0>\cdots>i_k\}$ and $J=\{j_0<\cdots<j_m\}$. If $x,y\in I$ or $x,y\in J$, by construction, the above sequence $S$ does not change any signs and if $x\in I$ and $y\in J$, the sequence $S$ changes sign exactly once. If $x,y\notin I\cup J$, the sequence $S$ equals the sequence $(x-y$, $\pi'(x)-\pi'(y)$, $u\pi'(x)-u\pi'(y)$, $w'u\pi'(x)-w'u\pi'(y))$, which changes sign at most once since the only nonreducedness of $w'u\pi'$ comes from $a$ and $a+1$. Without loss of generality, we now assume that $x=j_t\in J$ and $y\notin I\cup J$. Notice that $y\notin I\cup J$ means that $\pi'(y)=\pi''(y)$, $u\pi'(y)=u'\pi''(y)$ and $w'u\pi'(y)=w''u'\pi''(y)$. Assume for the sake of contradiction that the sequence $S$ changes signs at least twice. The rest is case work.

Case 1: $y<a$ and $S$ changes signs in $\pi''$ and $u'$. We have $y<j_t$, $\pi''(y)>\pi''(j_t)>\pi'(j_t)$ and $u'\pi''(y)<u'\pi''(x)=u\pi'(j_t)$, contradicting $u\pi'$ being reduced.

Case 2: $y<a$ and $S$ changes signs in $\pi''$ and $w''$ but not in $u'$. This means $w''u'\pi''(y)<b$. Together with $y<a$ and $\pi'(y)>\pi'(a+1)$, we see that in $w'u\pi'$, the wires $a+1$ and $y$ intersect more than once, so the sequence $a+1-y$, $\pi'(a+1)-\pi'(y)$, $u\pi'(a+1)-u\pi'(y)$, $w'u\pi'(a+1)-w'u\pi'(y)$ changes sign at least twice, a contradiction. 

Case 3: $y<a$ and $S$ changes signs in $w''$ and $u'$ but not in $\pi''$. Analogous to case 2, this implies that the wires $a+1$ and $y$ intersect twice in $w'u\pi'$ producing a contradiction.

Case 4: $y>a$ and $S$ changes signs in $\pi''$ and $w''$ but not in $u'$. Notice that $\pi''(j_t)=\pi'(j_{t+1})>\pi'(j_t)$ when $t\leq m-1$, and $\pi''(j_m)=\pi'(a)>\pi'(j_m)$. Thus, if $y<j_t$ and $\pi''(y)>\pi''(j_t)>\pi'(j_t)$, then $\pi'$ has more than one right descent, which is a contradiction. This means that the only possibility is that $y>j_t$ and $\pi'(j_t)<\pi''(y)<\pi''(j_t)$. Then $u'\pi''(y)<u'\pi''(j_t)$ and $w''u'\pi''(y)>w''u'\pi''(y)$. At the same time, $\pi'(y)>\pi'(j_t)$, $u\pi'(y)=u'\pi''(y)<u'\pi''(j_t)=u\pi'(y)$. And $w'u\pi'(y)=w''u'\pi''(y)>w''u'\pi''(j_t)=w''u\pi'(j_t)=w'u\pi'(j_{t+1})>w'u\pi'(j_{t+1})$. This says that $w'u$ is not reduced, by considering wires $\pi'(j_t)$ and $\pi'(y)$, which is a contradiction.

Case 5: $y>a$ and $S$ changes signs in $\pi''$ and $u'$. This is the key case that justifies our construction. As discussed in the last case, the fact $S$ changes signs in $\pi''$ implies that $y>j_t$ and $\pi'(j_t)<\pi''(y)<\pi''(j_t)$. Then $u'\pi''(y)>u'\pi''(j_t)$ so $u\pi'(y)>u\pi'(j_t)$. Say $t<m$ for now. Then $\pi''(j_t)=\pi'(j_{t+1})$. Because $\pi'$ has a single descent, $\pi'(j_t)<\pi'(y)<\pi'(j_{t+1})$ implies that $j_t<y<j_{t+1}$. Recall that to construct $J$, $j_{t+1}$ is the smallest $j'>j_t$ such that $\pi'(j')<\pi'(a)$ and $u\pi'(j')>u\pi'(j_t)$. But we can take such $j'=y<j_{t+1}$, contradict the definition of $J$. If $t=m$, the same argument goes through as $y$ would have been added to $J$.

Case 6: $y>a$ and $S$ changes signs in $u'$ and $w''$. The is exactly same as in case 5. Running the analysis from left to right, instead of from right to left as in the last case, gives the same contradiction to the construction of $J$.
\end{proof}

\section*{Acknowledgements}
We wish to thank Anders Bj\"{o}rner for alerting us to important references and Vic Reiner, Richard Stanley, and Igor Pak for helpful comments.  We are especially grateful to our advisor Alex Postnikov for his observation that separable elements may be related to faces of graph associahedra.

\bibliographystyle{plain}
\bibliography{separable2}
\end{document}